\documentclass{aptpub}

\authornames{P.R. Jelenkovi\'c and M. Olvera-Cravioto} % insert the authors here for use in running head
\shorttitle{Implicit Renewal Theory on Trees} % insert short title here for use in running head

\usepackage{graphicx}
\newcommand{\boldC}{\Pi}
\def\Indicator{\mathop{\hskip0pt{1}}\nolimits}

\begin{document}

\title{Implicit Renewal Theory and Power Tails on Trees} % insert title - use \\ if it requires more than one line.

\authorone[Columbia University]{Predrag R. Jelenkovi\'c } % Affiliation is just the name of your university or institution

\addressone{Department of Electrical Engineering, Columbia University, New York, NY 10027 \\ Supported by the NSF, grant no. CMMI-1131053} % Your postal address goes here.

\authortwo[Columbia University]{Mariana Olvera-Cravioto}

\addresstwo{Department of Industrial Engineering and Operations Research, Columbia University, New York, NY 10027 \\ Supported by the NSF, grant no. CMMI-1131053}

%%% NSF Grant acknowledgement
%\footnote{Partially supported by NSF Grant CMMI-1131053}

\begin{abstract}
We extend Goldie's (1991) Implicit Renewal Theorem to enable the analysis of recursions on weighted branching trees. We illustrate the developed method by deriving the power tail asymptotics of the distributions of the solutions $R$ to
$$R \stackrel{\mathcal{D}}{=} \sum_{i=1}^N C_i R_i + Q, \qquad R \stackrel{\mathcal{D}}{=} \left( \bigvee_{i=1}^N C_i R_i \right) \vee Q,$$
and similar recursions, where $(Q, N, C_1, C_2, \dots)$ is a nonnegative random vector with  $N  \in \{0, 1, 2, 3, \dots\} \cup \{\infty\}$, and $\{R_i\}_{i\in \mathbb{N}}$ are iid copies of $R$, independent of $(Q, N, C_1, C_2, \dots)$; here $\vee$ denotes the maximum operator.
\end{abstract}

\keywords{Implicit renewal theory; weighted branching processes;  multiplicative cascades; stochastic recursions; power laws; large deviations; stochastic fixed point equations} % insert keywords separated by a semicolon

\ams{60H25}{60J80;60F10;60K05} % insert the primary Maths Subject Classification number in the first bracket
         % and the secondary ams number(s) in the second bracket
         % e.g. \ams{60E20}{49G03;49F10}

\section{Introduction}

This paper is motivated by the study of the nonhomogeneous linear recursion 
\begin{equation} \label{eq:IntroLinear} 
R \stackrel{\mathcal{D}}{=} \sum_{i=1}^N C_i R_i + Q,
\end{equation}
where $(Q, N, C_1, C_2, \dots)$ is a nonnegative random vector with $N \in \mathbb{N} \cup \{\infty\}$, $\mathbb{N} = \{0, 1, 2, 3, \dots\}$, $P(Q >0) > 0$, and $\{R_i\}_{i \in \mathbb{N}}$ is a sequence of iid random variables, independent of $(Q, N, C_1, C_2, \dots)$, having the same distribution as $R$. This recursion appeared recently in the stochastic analysis of Google's PageRank algorithm, see \cite{Volk_Litv_08, Jel_Olv_09} and the references therein for the latest work in the area. These types of weighted recursions, also studied in the literature on weighted branching processes \cite{Rosler_93} and branching random walks \cite{Biggins_77}, are found in the probabilistic analysis of other algorithms as well \cite{Ros_Rus_01, Nei_Rus_04}, e.g., Quicksort algorithm \cite{Fill_Jan_01}. 

In order to study the preceding recursion in its full generality we extend the implicit renewal theory of Goldie \cite{Goldie_91} to cover recursions on trees. The extension of Goldie's theorem is presented in Theorem~\ref{T.NewGoldie} of Section~\ref{S.Renewal}. One of the observations that allows this extension is that an appropriately constructed measure on a weighted branching tree is a renewal measure, see Lemma \ref{L.RenewalMeasure} and equation \eqref{eq:RenewalMeasure}. In the remainder of the paper we apply the newly developed framework to analyze a number of linear and non-linear stochastic recursions on trees, starting with \eqref{eq:IntroLinear}. Note that the majority of the work in the rest of the paper goes into the application of the main theorem to specific problems. 

In this regard, in Section~\ref{S.LinearRec}, we first construct an explicit solution \eqref{eq:ExplicitConstr} to \eqref{eq:IntroLinear} on a weighted branching tree and then provide sufficient conditions for the finiteness of moments and the uniqueness of this solution in Lemmas \ref{L.Moments_R} and \ref{L.Convergence}, respectively. Furthermore, it is worth noting that our moment estimates are explicit, see Lemma~\ref{L.GeneralMoment}, which may be of independent interest. Then, the main result, which characterizes the power-tail behavior of $R$ is presented in Theorem \ref{T.LinearRecursion}. In addition, for integer power exponent ($\alpha \in \{ 1, 2, 3, \dots\}$) the asymptotic tail behavior can be explicitly computed as stated in Corollary \ref{C.explicit}. Furthermore, for non integer $\alpha$, Lemma \ref{L.Alpha_Moments} yields an explicit bound on the tail behavior of $R$.   Related work in the literature of weighted branching processes (WBPs) for the case when $N = \infty$ and $Q, \{C_i\}$ are nonnegative deterministic constants can be found in \cite{Rosler_93} (see Theorem 5), and more recently, for real valued constants, in \cite{Alsm_Rosl_05}. However, these deterministic assumptions fall outside of the scope of this paper; for more details see the remarks after Theorem~\ref{T.LinearRecursion} in Section~\ref{SS.MainLinear}.

Next, we show how our technique can be applied to study the tail asymptotics of the solution to the critical, $E\left[ \sum_{i=1}^N C_i \right] = 1$, homogeneous linear equation 
\begin{equation} \label{eq:IntroLinearHomog} 
R \stackrel{\mathcal{D}}{=} \sum_{i=1}^N C_i R_i,
\end{equation}
where $(N, C_1, C_2, \dots)$ is a nonnegative random vector with $N \in \mathbb{N} \cup \{\infty\}$ and $\{ R_i\}_{i\in \mathbb{N}}$ is a sequence of iid random variables independent of $(N, C_1, C_2, \dots)$ having the same distribution as $R$. This type of recursion has been studied to a great extent under a variety of names, including branching random walks and multiplicative cascades. Our work is more closely related to the results of \cite{Liu_00} and \cite{Iksanov_04}, where the conditions for power-tail asymptotics of the distribution of $R$ with power exponent $\alpha>1$ were derived.  In Theorem \ref{T.LinearHomog} of Section \ref{SS.MainLinear} we provide an alternative derivation of Theorem 2.2 in \cite{Liu_00} and Proposition 7 in \cite{Iksanov_04}. Furthermore, we note that our method yields a more explicit characterization of the power-tail proportionality constant, see Corollary~\ref{C.explicitHom}. For the full description of the set of solutions to \eqref{eq:IntroLinearHomog} see the very recent work in \cite{Als_Big_Mei_10}.  For additional references on weighted branching processes and multiplicative cascades see \cite{Alsm_Kuhl_07, Liu_00, Liu_98, Way_Will_95, Nei_Rus_04} and the references therein. For earlier historical references see \cite{Kah_Pey_76, Holl_Ligg_81, Durr_Ligg_83}.

As an additional illustration of the newly developed framework, in Section \ref{S.MaxRec} we study the recursion 
\begin{equation} \label{eq:IntroMaximum}
R \stackrel{\mathcal{D}}{=} \left(\bigvee_{i=1}^N C_i R_i \right) \vee Q,
\end{equation}
where $(Q, N, C_1, C_2, \dots)$ is a nonnegative random vector with $N \in \mathbb{N} \cup \{\infty\}$, \linebreak $P\left(Q >0 \right)>0$ and $\{R_i\}_{i\in \mathbb{N}}$ is a sequence of iid random variables independent of $(Q, N, C_1, C_2, \dots)$ having the same distribution as $R$. We characterize the tail behavior of $P(R > x)$ in Theorem \ref{T.MaximumRecursion}. Similarly to the homogeneous linear case, this recursion was previously studied in \cite{Alsm_Rosl_08} under the assumption that $Q \equiv 0$, $N = \infty$, and the $\{C_i\}$ are real valued deterministic constants. 
The more closely related case of $Q \equiv 0$ and $\{C_i \} \geq 0$ being random was studied earlier in \cite{Jag_Ros_04}.
Furthermore, these max-type stochastic recursions appear in a wide variety of applications, ranging from the average case analysis of algorithms to statistical physics; see \cite{Aldo_Band_05} for a recent survey.

We conclude the paper with a brief discussion of other non-linear recursions that could be studied using the developed techniques, including the solution to
$$R \stackrel{\mathcal{D}}{=} \left(\bigvee_{i=1}^N C_i R_i \right) + Q.$$
The majority of the proofs are postponed to Section~\ref{S.Proofs}.

%%% Model Description
\section{Model description} \label{S.ModelDescription}

First we construct a random tree $\mathcal{T}$. We use the notation $\emptyset$ to denote the root node of $\mathcal{T}$, and $A_n$, $n \geq 0$, to denote the set of all individuals in the $n$th generation of $\mathcal{T}$, $A_0 = \{\emptyset\}$. Let $Z_n$ be the number of individuals in the $n$th generation, that is, $Z_n = |A_n|$, where $| \cdot |$ denotes the cardinality of a set; in particular, $Z_0 = 1$. 

Next, let $\mathbb{N}_+ = \{1, 2, 3, \dots\}$ be the set of positive integers and let $U = \bigcup_{k=0}^\infty (\mathbb{N}_+)^k$ be the set of all finite sequences ${\bf i} = (i_1, i_2, \dots, i_n)$, where by convention $\mathbb{N}_+^0 = \{ \emptyset\}$ contains the null sequence $\emptyset$. To ease the exposition, for a sequence ${\bf i} = (i_1, i_2, \dots, i_k) \in U$ we write ${\bf i}|n = (i_1, i_2, \dots, i_n)$, provided $k \geq n$, and  ${\bf i}|0 = \emptyset$ to denote the index truncation at level $n$, $n \geq 0$. Also, for ${\bf i} \in A_1$ we simply use the notation ${\bf i} = i_1$, that is, without the parenthesis. Similarly, for ${\bf i} = (i_1, \dots, i_n)$ we will use $({\bf i}, j) = (i_1,\dots, i_n, j)$ to denote the index concatenation operation, if ${\bf i} = \emptyset$, then $({\bf i}, j) = j$. 

We iteratively construct the tree as follows. Let $N$ be the number of individuals born to the root node $\emptyset$, $N_\emptyset = N$, and let $\{N_{\bf i} \}_{{\bf i} \in U}$ be iid copies of $N$. Define now 
\begin{align} 
A_1 &= \{ i \in \mathbb{N}_+: 1 \leq i \leq N \}, \notag \\
A_n &= \{ (i_1, i_2, \dots, i_n) \in U: (i_1, \dots, i_{n-1}) \in A_{n-1}, 1 \leq i_n \leq N_{(i_1, \dots, i_{n-1})} \}. \label{eq:AnDef}
\end{align}
It follows that the number of individuals $Z_n = |A_n|$ in the $n$th generation, $n \geq 1$, satisfies the branching recursion 
$$Z_{n} = \sum_{{\bf i} \in A_{n-1}} N_{\bf i}.$$ 

Now, we construct the weighted branching tree $\mathcal{T}_{Q,C}$ as follows. The root node $\emptyset$ is assigned a vector $(Q_\emptyset, N_\emptyset, C_{(\emptyset, 1)}, C_{(\emptyset, 2)}, \dots) = (Q, N, C_1, C_2, \dots)$ with $N \in \mathbb{N} \cup \{\infty\}$ and $P(Q > 0) > 0$; $N$ determines the number of nodes in the first generation of $\mathcal{T}$ according to \eqref{eq:AnDef}. Each node in the first generation is then assigned an iid copy $(Q_i, N_i, C_{(i,1)}, C_{(i, 2)}, \dots)$ of the root vector and the $\{N_i\}$ are used to define the second generation in $\mathcal{T}$ according to \eqref{eq:AnDef}. 
In general, for $n\ge 2$, to each node ${\bf i} \in A_{n-1}$, we assign an iid copy 
$(Q_{\bf i}, N_{\bf i}, C_{({\bf i},1)}, C_{({\bf i}, 2)}, \dots)$ of the root vector and construct 
$A_{n} = \{({\bf i}, i_{n}): {\bf i} \in A_{n-1}, 1 \leq i_{n} \leq N_{\bf i}\}$;  the vectors $(Q_{\bf i}, N_{\bf i}, C_{({\bf i},1)}, C_{({\bf i}, 2)},  \dots)$, ${\bf i} \in A_{n-1}$ are chosen independently of all the 
previously assigned vectors  $(Q_{\bf j}, N_{\bf j}, C_{({\bf j},1)}, C_{({\bf j}, 2)}, \dots)$, ${\bf j} \in A_{k}, 0\le k\le n-2$.
For each node in $\mathcal{T}_{Q,C}$ we also define the weight $\boldC_{(i_1,\dots,i_n)}$ via the recursion
$$ \boldC_{i_1} =C_{i_1}, \qquad \boldC_{(i_1,\dots,i_n)} = C_{(i_1,\dots, i_n)} \boldC_{(i_1,\dots,i_{n-1})}, \quad n \geq 2,$$
where $\boldC =1$ is the weight of the root node. Note that the weight $\boldC_{(i_1,\dots, i_n)}$ is equal to the product of all the weights $C_{(\cdot)}$ along the branch leading to node $(i_1, \dots, i_n)$, as depicted in Figure \ref{F.Tree}.  
In some places, e.g. in the following section, the value of $Q$ may be of no importance, and thus we will consider a 
weighted branching tree defined by the smaller vector $(N, C_1, C_2, \dots)$.
This tree can be obtained form $\mathcal{T}_{Q,C}$ by simply disregarding the values for $Q_{(\cdot)}$ and is denoted by $\mathcal{T}_C$.

\begin{center}
\begin{figure}[h,t]
\begin{picture}(430,160)(0,0)
\put(0,0){\includegraphics[scale = 0.8, bb = 0 510 500 700, clip]{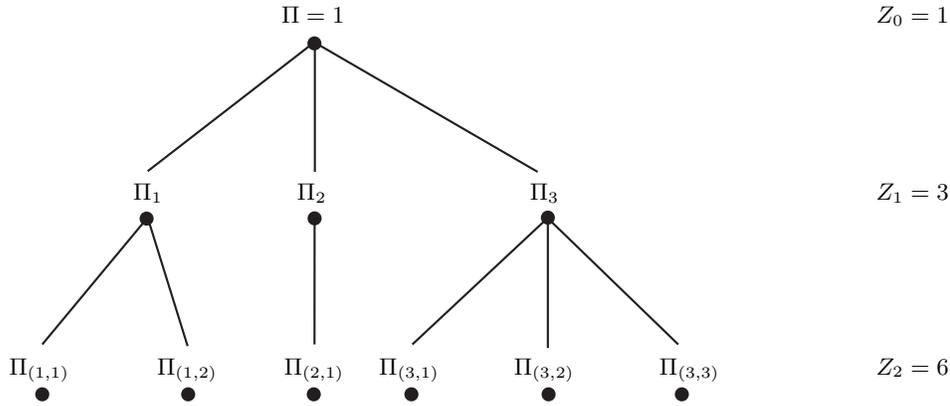}}
\put(125,150){\small $\boldC = 1$}
\put(69,83){\small $\boldC_{1}$}
\put(131,83){\small $\boldC_{2}$}
\put(219,83){\small $\boldC_{3}$}
\put(22,17){\small $\boldC_{(1,1)}$}
\put(78,17){\small $\boldC_{(1,2)}$}
\put(126,17){\small $\boldC_{(2,1)}$}
\put(162,17){\small $\boldC_{(3,1)}$}
\put(213,17){\small $\boldC_{(3,2)}$}
\put(268,17){\small $\boldC_{(3,3)}$}
\put(350,150){\small $Z_0 = 1$}
\put(350,83){\small $Z_1 = 3$}
\put(350,17){\small $Z_2 = 6$}
\end{picture}
\caption{Weighted branching tree}\label{F.Tree}
\end{figure}
\end{center}

Studying the tail behavior of the solutions to recursions and fixed point equations embedded in this weighted branching tree is the objective of this paper.

\section{Implicit renewal theorem on trees} \label{S.Renewal}

In this section we present an extension of Goldie's Implicit Renewal Theorem \cite{Goldie_91} to weighted branching trees. The observation that facilitates this generalization is the following lemma which shows that a certain measure on a tree is actually a product measure; a similar measure was used in a different context in \cite{Biggins_Kyprianou_77}. Its proof is given in Section~\ref{SS.ImplicitProofs} for completeness.  Throughout the paper we use the standard convention $0^\alpha \log 0 = 0$ for all $\alpha > 0$.

\begin{lem} \label{L.RenewalMeasure}
Let  $\mathcal{T}_{C}$ be the weighted branching tree defined by the nonnegative vector $(N, C_1, C_2, \dots)$, where $N \in \mathbb{N} \cup \{\infty\}$.  For any $n \in \mathbb{N}$ and ${\bf i} \in A_n$,
let $V_{{\bf i}} = \log \boldC_{\bf i}$. For $\alpha > 0$ define the measure
$$\mu_n(dt) = e^{\alpha t} E\left[  \sum_{{\bf i} \in A_n} \Indicator(V_{{\bf i}} \in dt )  \right], \quad n = 1, 2, \dots,$$
and let $\eta(dt) = \mu_1(dt)$. 
Suppose that there exists $j \geq 1$ with $P(N\ge j,C_j>0)~>~0$ such that the measure $P(\log C_j\in du, C_j > 0, N\ge j)$ is nonarithmetic, \linebreak $0 < E\left[ \sum_{i=1}^N C_i^\alpha \log C_i  \right] ~< ~\infty$ and $E\left[ \sum_{i=1}^N C_i^\alpha \right] = 1$. Then, $\eta(\cdot)$ is a nonarithmetic probability measure on $\mathbb{R}$ that places no mass at $-\infty$ and has mean
$$\int_{-\infty}^\infty u\, \eta(du) = E\left[ \sum_{j=1}^N C_j^\alpha \log C_j  \right] .$$
Furthermore, $\mu_n(dt) = \eta^{*n}(dt)$, where $\eta^{*n}$ denotes the $n$th convolution of $\eta$ with itself.
\end{lem}

We now present a generalization of Goldie's Implicit Renewal Theorem \cite{Goldie_91} that will enable the analysis of recursions on weighted branching trees. Note that except for the independence assumption, the random variable $R$ and the vector $(N, C_1, C_2, \dots)$ are arbitrary, and therefore the applicability of this theorem goes beyond the recursions that we study here. Throughout the paper we use $g(x) \sim f(x)$ as $x \to \infty$ to denote $\lim_{x \to \infty} g(x)/f(x) = 1$.

\begin{thm} \label{T.NewGoldie}
Let $(N, C_1, C_2, \dots)$ be a nonnegative random vector, where $N \in \mathbb{N} \cup \{\infty\}$. 
Suppose that there exists $j \geq 1$ with $P(N\ge j,C_j>0)>0$ such that the measure $P(\log C_j\in du, C_j > 0, N\ge j)$ is nonarithmetic.
Assume further that \linebreak 
$0 <  E\left[ \sum_{j=1}^N C_j^\alpha \log C_j  \right] < \infty$, $E\left[ \sum_{j=1}^N C_j^\alpha \right] = 1$, $E\left[ \sum_{j=1}^N C_j^\gamma \right] < \infty$ for some $0 \leq \gamma < \alpha$, and that $R$ is independent of $(N, C_1, C_2, \dots)$ with $E[R^\beta] < \infty$ for any $0< \beta < \alpha$. If 
\begin{equation} \label{eq:Goldie_condition}
\int_0^\infty \left| P(R > t) - E\left[ \sum_{j=1}^N \Indicator (C_j R > t ) \right] \right| t^{\alpha-1} dt < \infty,
\end{equation}
then
$$P(R > t) \sim H t^{-\alpha}, \qquad t \to \infty,$$
where $0 \leq H < \infty$ is given by
\begin{align*}
H &= \frac{1}{E\left[ \sum_{j=1}^N C_j^\alpha \log C_j  \right] } \int_{0}^\infty v^{\alpha-1} \left( P(R > v) - E\left[ \sum_{j=1}^{N} \Indicator(C_{j} R > v ) \right]    \right) dv .
\end{align*}
\end{thm}

{\sc Remarks:} (i) As pointed out in \cite{Goldie_91}, the statement of the theorem only has content when $R$ has infinite moment of order $\alpha$, since otherwise the constant $H$ is zero. (ii) Similarly as in \cite{Goldie_91}, this theorem can be generalized to incorporate negative weights $\{C_i\}$ at the expense of additional technical complications. However, when the $\{C_i\} \geq 0$ and $R$ is real-valued, one can use exactly the same proof to derive the asymptotics of $P(-R > t)$; we omit the statement here since our applications do not require it. (iii) When the $\{\log C_i\}$ are lattice valued, a similar version of the theorem can be derived by using the corresponding Renewal Theorem for lattice random walks. (iv) It appears, as noted in \cite{Goldie_91}, that some of the early ideas of applying renewal theory to study the power tail asymptotics of autoregressive processes (perpetuities) is due to \cite{Kesten_73} and \cite{Grincevicius_75}. The proof given below follows the corresponding proof in \cite{Goldie_91}.

%Before going into the proof of Theorem \ref{T.NewGoldie} we need the following monotone density lemma which we prove in Section \ref{SS.ImplicitProofs} for completeness. 

%
%\begin{lem} \label{L.Derivative}
%Let $\alpha, \beta > 0$ and $0 \leq H < \infty$. Suppose $\int_0^t v^{\alpha+\beta-1} P(R > v) dv \sim H t^{\beta}/\beta$ as $t \to \infty$. Then,
%$$P(R > t) \sim H t^{-\alpha}, \qquad t \to \infty.$$
%\end{lem}

\begin{proof}[Proof of Theorem \ref{T.NewGoldie}]
Let  $\mathcal{T}_{C}$ be the weighted branching tree defined by the nonnegative vector $(N, C_1, C_2, \dots)$.  For each ${\bf i} \in A_n$ and all $k \leq n$ define $V_{{\bf i}|k} = \log \boldC_{{\bf i}|k}$; note that $\boldC_{{\bf i}|k}$ is independent of $N_{{\bf i}|k}$ but not of $N_{{\bf i}|s}$ for any $0\leq s \leq k-1$. Also note that ${\bf i}|n = {\bf i}$ since ${\bf i} \in A_n$. Let  $\mathcal{F}_k$, $k \geq 1$, denote the $\sigma$-algebra generated by $\left\{ (N_{\bf i}, C_{({\bf i}, 1)}, C_{({\bf i}, 2)}, \dots) : {\bf i} \in A_j, 0 \leq j \leq k-1 \right\}$, and let $\mathcal{F}_0 = \sigma(\emptyset, \Omega)$, $\boldC_{{\bf i}|0} \equiv 1$.  Assume also that $R$ is independent of the entire weighted tree, $\mathcal{T}_{C}$. Then, for any $t \in \mathbb{R}$, we can write $P(R > e^t)$ via a telescoping sum as follows (note that all the expectations in \eqref{eq:telescoping} are finite by Markov's inequality and \eqref{eq:PiMoments})
\begin{align}
&P(R > e^t) \notag \\
&= \sum_{k=0}^{n-1} \left( E\left[ \sum_{({\bf i}|k) \in A_{k}} \Indicator(\boldC_{{\bf i}|k} R > e^t ) \right] - E\left[ \sum_{({\bf i}|k+1) \in A_{k+1}} \Indicator(\boldC_{{\bf i}|k+1} R > e^t) \right]  \right) \label{eq:telescoping} \\
&\hspace{5mm} + E\left[ \sum_{({\bf i}|n) \in A_n} \Indicator(\boldC_{{\bf i}|n} R > e^t ) \right] \notag \\
&= \sum_{k=0}^{n-1}  E\left[ \sum_{({\bf i}|k) \in A_{k}} \left( \Indicator (\boldC_{{\bf i}|k} R > e^t) -  \sum_{j=1}^{N_{{\bf i}|k}} \Indicator(\boldC_{{\bf i}|k} C_{({\bf i}|k,j)} R > e^t) \right) \right]  \notag \\
&\hspace{5mm} + E\left[ \sum_{({\bf i}|n) \in A_n} \Indicator(\boldC_{{\bf i}|n} R > e^t ) \right] \notag \\
&= \sum_{k=0}^{n-1} E\left[ \sum_{({\bf i}|k) \in A_{k}} E\left[ \left.   \Indicator( R > e^{t-V_{{\bf i}|k}} ) -  \sum_{j=1}^{N_{{\bf i}|k}} \Indicator( C_{({\bf i}|k,j)} R > e^{t-V_{{\bf i}|k}} )  \right| \mathcal{F}_k  \right] \right]  \notag \\
&\hspace{5mm} + E\left[ \sum_{({\bf i}|n) \in A_n} \Indicator(\boldC_{{\bf i}|n} R > e^t ) \right] .
\end{align}

Now, define the measures $\mu_n$ according to Lemma \ref{L.RenewalMeasure} and let
$$\nu_n(dt) = \sum_{k=0}^n \mu_k(dt), \qquad g(t) = e^{\alpha t} \left( P(R > e^t) - E\left[ \sum_{j=1}^{N} \Indicator(C_{j} R > e^{t} ) \right]    \right),$$
$$r(t) = e^{\alpha t} P(R > e^t) \qquad \text{and} \qquad \delta_n(t) = e^{\alpha t} E\left[ \sum_{({\bf i}|n) \in A_n} \Indicator( \boldC_{{\bf i}|n} R > e^t ) \right].$$

Recall that $R$ and $(N_{{\bf i}|k}, C_{({\bf i}|k,1)}, C_{({\bf i}|k,2)}, \dots)$ are independent of $\mathcal{F}_k$, from where it follows that
$$E\left[ \left.   \Indicator( R > e^{t-V_{{\bf i}|k}} ) -  \sum_{j=1}^{N_{{\bf i}|k}} \Indicator( C_{({\bf i}|k,j)} R > e^{t-V_{{\bf i}|k}} )  \right| \mathcal{F}_k  \right]  = e^{\alpha(V_{{\bf i}|k}-t)} g\left( t-V_{{\bf i}|k}  \right). $$
Then, for any $t \in \mathbb{R}$ and $n \in \mathbb{N}$, 
\begin{align*}
r(t) &= \sum_{k=0}^{n-1} E\left[ \sum_{({\bf i}|k) \in A_k} e^{\alpha V_{{\bf i}|k}} g(t-V_{{\bf i}|k}) \right] + \delta_n(t) = (g*\nu_{n-1})(t) + \delta_n(t).
\end{align*}

Next, define the operator $\breve{f}(t) = \int_{-\infty}^t e^{-(t-u)} f(u) \, du$ and note that 
\begin{equation} \label{eq:SmoothOperator}
\breve{r}(t) = (\breve{g}* \nu_{n-1})(t) + \breve{\delta}_n(t).
\end{equation} 

%\begin{align}
%\breve{r}(t) &= \int_{-\infty}^t e^{-\beta(t-u)} (g*\nu_{n-1})(u) \, du + \breve{\delta}_n (t) \notag \\
%&= \int_{-\infty}^t e^{-\beta(t-u)} \int_{-\infty}^\infty g(u-v) \nu_{n-1}(dv) \, du + \breve{\delta}_n (t) \notag \\
%&= \int_{-\infty}^\infty \int_{-\infty}^t e^{-\beta(t-u)} g(u-v) \, du \, \nu_{n-1}(dv) + \breve{\delta}_n (t) \notag \\
%&= \int_{-\infty}^\infty \breve{g}(t-v) \, \nu_{n-1}(dv) + \breve{\delta}_n (t) \notag \\
%&= (\breve{g}* \nu_{n-1})(t) + \breve{\delta}_n(t) . \label{eq:SmoothOperator}
%\end{align}

Now, we will show that one can let $n \to \infty$ in the preceding identity. To this end, let $\eta(du) = \mu_1(du)$, and note that by Lemma \ref{L.RenewalMeasure} $\eta(\cdot)$ is a nonarithmetic probability measure on $\mathbb{R}$ that places no mass at $-\infty$ and has mean,
$$\mu \triangleq \int_{-\infty}^\infty u\, \eta(du) = E\left[ \sum_{j=1}^N C_j^\alpha \log C_j  \right] > 0 .$$
Moreover, by Lemma \ref{L.RenewalMeasure},
\begin{equation} \label{eq:RenewalMeasure}
\nu(dt) \triangleq  \sum_{k=0}^\infty e^{\alpha t} E\left[ \sum_{({\bf i}|k) \in A_{k}}  \Indicator(V_{{\bf i}|k} \in dt )   \right]   = \sum_{k=0}^\infty \eta^{*k}(dt)
\end{equation}
is its renewal measure. Since $\mu \neq 0$, then $(|f|*\nu)(t) < \infty$ for all $t$ whenever $f$ is directly Riemann integrable. By \eqref{eq:Goldie_condition} we know that $g \in L_1$, so by Lemma 9.1 from \cite{Goldie_91}, $\breve{g}$ is directly Riemann integrable, resulting in $(|\breve{g}|*\nu)(t) < \infty$ for all $t$.  Thus, $(|\breve{g}|*\nu)(t) =  E\left[ \sum_{k=0}^\infty  \sum_{({\bf i}|k) \in A_{k}} e^{\alpha V_{{\bf i}|k}} |\breve{g}(t-V_{{\bf i}|k})| \right]  < \infty$, which implies that $E\left[ \sum_{k=0}^\infty  \sum_{({\bf i}|k) \in A_{k}} e^{\alpha V_{{\bf i}|k}} \breve{g}(t-V_{{\bf i}|k}) \right]$ exists and, by Fubini's theorem, 
\begin{align*}
(\breve{g}*\nu)(t) &= E\left[ \sum_{k=0}^\infty  \sum_{({\bf i}|k) \in A_{k}} e^{\alpha V_{{\bf i}|k}} \breve{g}(t-V_{{\bf i}|k}) \right]  \\
&= \sum_{k=0}^\infty E\left[ \sum_{({\bf i}|k) \in A_{k}} e^{\alpha V_{{\bf i}|k}} \breve{g}(t-V_{{\bf i}|k}) \right]  = \lim_{n\to \infty}  (\breve{g}*\nu_n)(t).
\end{align*}

To see that $\breve{\delta}_n(t) \to 0$ as $n \to \infty$ for all fixed $t$, note that from the assumptions $0 < E\left[ \sum_{j=1}^N C_j^\alpha \log C_j \right] < \infty$, $E\left[ \sum_{j=1}^N C_j^\alpha \right] = 1$, and $E\left[ \sum_{j=1}^N C_j^\gamma \right] < \infty$ for some $0 \leq \gamma < \alpha$, there exists $0 < \beta <\alpha$ such that $E\left[ \sum_{j=1}^N C_j^{\beta} \right] < 1$ (by convexity). Then, for such $\beta$, 
\begin{align}
\breve{\delta}_n(t) &= \int_{-\infty}^t e^{-(t-u)} e^{\alpha u} E\left[ \sum_{({\bf i}|n) \in A_n} \Indicator\left( \boldC_{{\bf i}|n} R > e^{u} \right) \right]  du \notag \\
&\leq e^{(\alpha-\beta)t}  E\left[ \sum_{({\bf i}|n) \in A_n}  \int_{-\infty}^t e^{ \beta u} \Indicator\left(\boldC_{{\bf i}|n}  R > e^u \right)  du \, \right]  \notag \\
&= e^{(\alpha-\beta) t}  E\left[ \sum_{({\bf i}|n) \in A_n}  \int_{-\infty}^{\min\{t, \log(\boldC_{{\bf i}|n} R)\}}  e^{\beta u}   du \, \right]  \notag \\
&\leq \frac{e^{(\alpha-\beta) t}}{\beta} E\left[ \sum_{({\bf i}|n) \in A_n}   (\boldC_{{\bf i}|n} R)^{\beta}   \right]  .  \label{eq:delta_error}
\end{align}
It remains to show that the expectation in \eqref{eq:delta_error} converges to zero as $n \to \infty$. First note that from the independence of $R$ and $\mathcal{T}_C$, 
$$ E\left[ \sum_{({\bf i}|n) \in A_n}   (\boldC_{{\bf i}|n} R)^{\beta}   \right]  = E[R^\beta]  E\left[ \sum_{({\bf i}|n) \in A_n}   (\boldC_{{\bf i}|n})^{\beta}   \right],$$
where $E[R^\beta] < \infty$, for $0 < \beta < \alpha$. For the expectation involving $\boldC_{{\bf i}|n}$ condition on $\mathcal{F}_{n-1}$ and use the independence of $(N_{{\bf i}|n-1}, C_{({\bf i}|n-1, 1)}, C_{({\bf i}|n-1, 2)}, \dots)$ from $\mathcal{F}_{n-1}$ as follows
\begin{align}
E\left[   \sum_{({\bf i}|n) \in A_n}  (\boldC_{{\bf i}|n})^{\beta}  \right] &= E\left[   \sum_{({\bf i}|n-1) \in A_{n-1}} E\left[ \left. \sum_{j=1}^{N_{{\bf i}|n-1}}   (\boldC_{{\bf i}|n-1})^{\beta} C_{({\bf i}|n-1,j)}^{\beta} \right| \mathcal{F}_{n-1} \right] \right] \notag \\
&= E\left[   \sum_{({\bf i}|n-1) \in A_{n-1}} (\boldC_{{\bf i}|n-1})^{\beta} E\left[ \left. \sum_{j=1}^{N_{{\bf i}|n-1}}    C_{({\bf i}|n-1,j)}^{\beta} \right| \mathcal{F}_{n-1} \right] \right] \notag \\
&= E\left[  \sum_{j=1}^N    C_j^{\beta}  \right]  E\left[   \sum_{({\bf i}|n-1) \in A_{n-1}} (\boldC_{{\bf i}|n-1})^{\beta} \right] \notag \\
&= \left( E\left[     \sum_{j=1}^{N}  C_{j}^{\beta}   \right] \right)^n \qquad \text{(iterating $n-1$ times)}. \label{eq:PiMoments}
\end{align}
Since $E\left[     \sum_{j=1}^{N}  C_{j}^{\beta}   \right] < 1$, then the above converges to zero as $n \to \infty$. Hence, the preceding arguments allow us to pass $n \to \infty$ in \eqref{eq:SmoothOperator}, and obtain
$$\breve{r}(t) = (\breve{g}*\nu)(t).$$

Now, by the key renewal theorem for two-sided random walks, see Theorem 4.2 in \cite{Ath_McD_Ney_78},
$$e^{- t} \int_{0}^{e^t} v^{\alpha} P(R > v) \, dv = \breve{r}(t) \to \frac{1}{\mu} \int_{-\infty}^\infty \breve{g}(u) \, du \triangleq H, \qquad t \to \infty.$$
Clearly, $H \geq 0$ since the left-hand side of the preceding equation is positive, and thus, by 
Lemma 9.3 in \cite{Goldie_91},
$$P(R > t) \sim H t^{-\alpha}, \qquad t \to \infty.$$

Finally, 
\begin{align*}
H &=  \frac{1}{\mu} \int_{-\infty}^\infty \int_{-\infty}^u e^{-(u-t)} g(t) \, dt \, du \\
&=  \frac{1}{\mu} \int_{-\infty}^\infty e^{ t} g(t)  \int_{t}^\infty  e^{- u} \, du \, dt \\
&= \frac{1}{ \mu} \int_{-\infty}^\infty  g(t) \, dt \\
&= \frac{1}{ \mu} \int_{-\infty}^\infty e^{\alpha t} \left( P(R > e^t) - E\left[ \sum_{j=1}^{N} \Indicator(C_{j} R > e^{t} ) \right]    \right) dt \\
&= \frac{1}{ \mu} \int_{0}^\infty v^{\alpha-1} \left( P(R > v) - E\left[ \sum_{j=1}^{N} \Indicator(C_{j} R > v ) \right]    \right) dv. 
\end{align*}
\end{proof}

\vspace{-15pt}

%%%%%%%%% The linear recursion
\section{The linear recursion: $R = \sum_{i=1}^N C_i R_i + Q$} \label{S.LinearRec}

Motivated by the information ranking problem on the Internet, e.g. Google's PageRank algorithm \cite{Jel_Olv_09, Volk_Litv_08}, in this section we apply the implicit renewal theory for trees developed in the previous section to the following linear recursion:
\begin{equation} \label{eq:Linear}
R \stackrel{\mathcal{D}}{=} \sum_{i=1}^N C_i R_i + Q,
\end{equation}
where $(Q,N, C_1, C_2, \dots)$ is a nonnegative random vector with $N \in \mathbb{N} \cup \{\infty\}$, \linebreak $P(Q > 0 ) > 0$, and $\{R_i\}_{i\in \mathbb{N}}$ is a sequence of iid random variables independent of $(Q, N, C_1, C_2, \dots)$ having the same distribution as $R$. Note that the power tail of $R$ in the critical homogeneous case $(Q \equiv 0)$ was previously studied in \cite{Liu_00} and \cite{Iksanov_04}. In Section~\ref{SS.Homogeneous} we will give an alternative derivation of those results using our method and will provide pointers to the appropriate literature. 

As for the nonhomogeneous case, the first result we need to establish is the existence and finiteness of a solution to \eqref{eq:Linear}. For the purpose of existence we will provide an explicit construction of the solution $R$ to \eqref{eq:Linear} on a tree.  Note that such constructed $R$ will be the main object of study of this section. 

Recall that throughout the paper the convention is to denote the random vector associated to the root node  $\emptyset$ by $(Q, N, C_1, C_2, \dots) \equiv (Q_\emptyset, N_\emptyset, C_{(\emptyset, 1)}, C_{(\emptyset, 2)}, \dots)$. 

We now define the process
\begin{equation} \label{eq:W_k}
W_0 =  Q, \quad W_n =  \sum_{{\bf i} \in A_n} Q_{{\bf i}} \boldC_{{\bf i}}, \qquad n \geq 1,
\end{equation}
on the weighted branching tree $\mathcal{T}_{Q, C}$, as constructed in Section \ref{S.ModelDescription}. Define the process $\{R^{(n)}\}_{n \geq 0}$ according to
\begin{equation} \label{eq:R_nDef}
R^{(n)} = \sum_{k=0}^n W_k , \qquad n \geq 0,
\end{equation}
that is, $R^{(n)}$ is the sum of the weights of all the nodes on the tree up to the $n$th generation. It is not hard to see that $R^{(n)}$ satisfies the recursion
\begin{equation} \label{eq:LinearRecSamplePath} 
R^{(n)} = \sum_{j=1}^{N_\emptyset} C_{(\emptyset,j)} R^{(n-1)}_{j} + Q_{\emptyset} = \sum_{j=1}^{N} C_{j} R^{(n-1)}_{j} + Q, \qquad n \geq 1,
\end{equation}
where $\{R_{j}^{(n-1)}\}$ are independent copies of $R^{(n-1)}$ corresponding to the tree starting with individual $j$ in the first generation and ending on the $n$th generation; note that $R_j^{(0)} = Q_j$. Similarly, since the tree structure repeats itself after the first generation, $W_n$ satisfies
\begin{align}
W_n &= \sum_{{\bf i} \in A_n} Q_{{\bf i}} \boldC_{{\bf i}} \notag\\
&= \sum_{k = 1}^{N_\emptyset} C_{(\emptyset,k)}  
\sum_{(k,\dots,i_n) \in A_n} Q_{(k,\dots,i_n)} \prod_{j=2}^n C_{(k,\dots,i_j)}  \notag\\
&\stackrel{\mathcal{D}}{=} \sum_{k=1}^N C_k W_{(n-1),k},\label{eq:WnRec}
\end{align}
where $\{W_{(n-1),k}\}$ is a sequence of iid random variables independent of $(N, C_1, C_2, \dots)$ and having the same distribution as $W_{n-1}$. 

Next, define the random variable $R$ according to
\begin{equation} \label{eq:ExplicitConstr}
R \triangleq \lim_{n\to \infty} R^{(n)} = \sum_{k=0}^\infty W_k ,
\end{equation}
where the limit is properly defined by \eqref{eq:R_nDef} and monotonicity. Hence, it is easy to verify, by applying monotone convergence in \eqref{eq:LinearRecSamplePath}, that $R$ must solve
$$R = \sum_{j=1}^{N_\emptyset} C_{(\emptyset,j)} R_{j}^{(\infty)} + Q_\emptyset = \sum_{j=1}^N C_{j} R_j^{(\infty)} + Q ,$$
where $\{R_j^{(\infty)}\}_{j \in \mathbb{N}}$ are iid, have the same distribution as $R$, and are independent of $(Q, N, C_1, C_2, \dots)$. 

The derivation provided above implies in particular the existence of a solution in distribution to \eqref{eq:Linear}. Moreover, under additional technical conditions, $R$ is the unique solution under iterations as we will define and show in the following section. The constructed $R$, as defined in \eqref{eq:ExplicitConstr}, is the main object of study in the remainder of this section.

\subsection{Moments of $W_n$ and $R$} \label{SS.MomentsLinear}

In this section we derive estimates for the moments of $W_n$ and $R$. We start by stating a lemma about the moments of a sum of random variables.  The proofs of Lemmas \ref{L.Alpha_Moments}, \ref{L.MomentSmaller_1} and \ref{L.GeneralMoment} are given in Section~\ref{SS.MomentsProofs}. 

\begin{lem} \label{L.Alpha_Moments}
For any $k \in \mathbb{N} \cup \{\infty\}$ let $\{C_i \}^k_{i=1}$ be a sequence of nonnegative random variables and let $\{Y_i\}_{i = 1}^k$ be a sequence of nonnegative iid random variables, independent of the $\{C_i\}$, having the same distribution as $Y$. For $\beta > 1$ set $p = \lceil \beta \rceil \in \{2, 3, 4, \dots\}$, and if $k = \infty$ assume that $\sum_{i=1}^\infty C_i Y_i < \infty$ a.s. Then,
$$E\left[ \left( \sum_{i=1}^k C_i Y_i \right)^\beta - \sum_{i=1}^k (C_iY_i)^\beta \right] \leq \left( E\left[ Y^{p-1} \right] \right)^{\beta/(p-1)} E\left[ \left(\sum_{i=1}^k C_i \right)^\beta  \right].$$
\end{lem}

{\sc Remark:} Note that the preceding lemma does not exclude the case when \linebreak $E \left[ \left( \sum_{i=1}^k C_i Y_i \right)^\beta  \right] = \infty$ but $E\left[ \left( \sum_{i=1}^k C_i Y_i \right)^\beta - \sum_{i=1}^k (C_iY_i)^\beta \right] < \infty$. 

\bigskip

We now give estimates for the $\beta$-moments of $W_n$ for $\beta \in (0, 1]$ and $\beta > 1$ in Lemmas~\ref{L.MomentSmaller_1} and \ref{L.GeneralMoment}, respectively. Throughout the rest of the paper define \linebreak $\rho_\beta = E\left[ \sum_{i=1}^N C_i^\beta \right]$ for any $\beta > 0$, and $\rho \equiv \rho_1$.

\begin{lem} \label{L.MomentSmaller_1}
For $0 < \beta \leq 1$ and all $n \geq 0$,
$$E[ W_n^\beta ] \leq  E[Q^\beta] \rho_\beta^{n}.$$
\end{lem}

\begin{lem} \label{L.GeneralMoment}
For $\beta > 1$ suppose $E[Q^\beta]< \infty$, $E\left[ \left( \sum_{i=1}^N C_i \right)^\beta \right] < \infty$, and $\rho \vee \rho_\beta < 1$. Then, there exists a constant $K_\beta > 0$ such that for all $n \geq 0$, 
\begin{equation*} 
E[ W_n^\beta ] \leq K_\beta ( \rho \vee \rho_\beta  )^{n}.
\end{equation*}
\end{lem}

Now we are ready to establish the finiteness of moments of the solution $R$ given by \eqref{eq:ExplicitConstr}  in Section \ref{S.LinearRec}. The proof of this lemma uses well known contraction arguments, but for completeness we provide the details below.

\begin{lem} \label{L.Moments_R}
Assume that $E[Q^\beta] < \infty$ for some $\beta > 0$. In addition, suppose that either (i) $\rho_\beta < 1$ if $0 < \beta < 1$, or (ii) $(\rho \vee \rho_\beta)  < 1$ and $E\left[\left( \sum_{i=1}^N C_i \right)^\beta\right] < \infty$ if $\beta \geq 1$. Then, $E[R^\gamma] < \infty$ for all $0  < \gamma \leq \beta$, and in particular, $R < \infty$ a.s. Moreover, if $\beta \geq 1$, $R^{(n)} \stackrel{L_\beta}{\to} R$, where $L_\beta$ stands for convergence in $(E|\cdot|^\beta)^{1/\beta}$ norm. 
\end{lem}

{\sc Remark:} It is interesting to observe that for $\beta > 1$ the conditions $\rho_\beta < 1$ and $E\left[ \left( \sum_{i=1}^N C_i \right)^\beta \right] < \infty$ are consistent with Theorem 3.1 in \cite{Alsm_Kuhl_07}, Proposition 4 in \cite{Iksanov_04} and Theorem 2.1 in \cite{Liu_00}, 
which give the conditions for the finiteness of the $\beta$-moment of the solution to the related critical ($\rho_1 = 1$) homogeneous ($Q \equiv 0$) equation.

\begin{proof}
Let 
$$\eta = \begin{cases} \rho_\beta & \text{ if }\beta < 1 \\ \rho \vee \rho_\beta, & \text{ if } \beta \geq 1. \end{cases}$$ 
Then by Lemmas \ref{L.MomentSmaller_1} and \ref{L.GeneralMoment}, 
\begin{equation} \label{eq:EW_n}
E[W_n^\beta] \leq K \eta^n 
\end{equation}
for some $K > 0$. Suppose $\beta \geq 1$, then, by monotone convergence and Minkowski's inequality, 
\begin{align*}
E[R^\beta] &= E\left[ \lim_{n\to\infty} \left(\sum_{k=0}^n W_k \right)^\beta  \right] = \lim_{n\to \infty} E\left[ \left(\sum_{k=0}^n W_k\right)^\beta  \right] \\
&\leq \lim_{n\to\infty} \left( \sum_{k=0}^n \left( E[W_k^\beta] \right)^{1/\beta} \right)^\beta \leq K \left( \sum_{k=0}^\infty \eta^{k/\beta}   \right)^\beta < \infty.
\end{align*}
This implies that $R < \infty$ a.s. When $0 < \beta \le 1$ use the inequality $\left( \sum_{k=0}^n y_k \right)^\beta \leq \sum_{k=0}^n y_k^\beta$ for any $y_i \geq 0$ instead of Minkowski's inequality. Furthermore, for any \linebreak $0 < \gamma \leq \beta$, 
$$E[R^\gamma] = E\left[ (R^\beta)^{\gamma/\beta}\right] \leq \left(E[R^\beta] \right)^{\gamma/\beta} < \infty.$$

That $R^{(n)} \stackrel{L_\beta}{\to} R$ whenever $\beta \geq 1$ follows from noting that $E[|R^{(n)} - R|^\beta]$ \linebreak $= E\left[ \left( \sum_{k = n+1}^\infty W_k \right)^\beta \right]$ and applying the same arguments used above to obtain the bound $E[|R^{(n)} - R|^\beta] \leq K \eta^{n+1}/(1 - \eta^{1/\beta})^\beta$. 
\end{proof}

Next, we show that under some technical conditions, the iteration of recursion \eqref{eq:Linear} results in a process that converges in distribution to $R$ for any initial condition $R_0^*$. To this end, consider a weighted branching tree $\mathcal{T}_{Q,C}$, as defined in Section \ref{S.ModelDescription}. Now, define 
$$R_n^*  \triangleq R^{(n-1)} + W_n(R_0^*), \qquad n \geq 1,$$
where $R^{(n-1)}$ is given by \eqref{eq:R_nDef},
\begin{equation} \label{eq:LastWeights}
W_n(R_0^*) = \sum_{{\bf i} \in A_n} R_{0,{\bf i}}^* \boldC_{\bf i},
\end{equation}
and $\{ R_{0,{\bf i}}^*\}_{{\bf i} \in U}$ are iid copies of an initial value $R_0^*$, independent of the entire weighted tree $\mathcal{T}_{Q,C}$. It follows from \eqref{eq:LinearRecSamplePath} and \eqref{eq:LastWeights} that, for $n \geq 0$, 
\begin{equation} \label{eq:StarRecursion}
R_{n+1}^* = \sum_{j=1}^{N} C_j R_{j}^{(n-1)}  + Q + W_{n+1}(R_0^*) = \sum_{j=1}^{N} C_j \left( R_{j}^{(n-1)}  + \sum_{{\bf i} \in A_{n,j}} R_{0,{\bf i}}^* \prod_{k=2}^n C_{(j,\dots,i_k)} \right) + Q,
\end{equation}
where $\{ R_{j}^{(n-1)} \}$ are independent copies of $R^{(n-1)}$ corresponding to the tree starting with individual $j$ in the first generation and ending on the $n$th generation, and $A_{n,j}$ is the set of all nodes in the $(n+1)$th generation that are descendants of individual $j$ in the first generation.  It follows that
$$R_{n+1}^*= \sum_{j=1}^N C_j R_{n,j}^* + Q,$$
where $\{R_{n,j}^*\}$ are the expressions inside the parenthesis in \eqref{eq:StarRecursion}. Clearly, $\{R_{n,j}^*\}$ are iid copies of $R_n^*$, thus we show that $R_n^*$ is equal in distribution to the process derived by iterating \eqref{eq:Linear} with an initial condition $R_0^*$. The following lemma shows that $R_n^* \Rightarrow R$ for any initial condition $R_0^*$ satisfying a moment assumption, where $\Rightarrow$ denotes convergence in distribution.

\begin{lem} \label{L.Convergence}
For any initial condition $R_0^* \geq 0$, if $E[Q^\beta], E[(R_0^*)^\beta] < \infty$ and $\rho_\beta = E\left[ \sum_{i=1}^N C_i^\beta \right] < 1$ for some $0 < \beta \leq 1$, then
$$R_n^* \Rightarrow R,$$
with $E[R^\beta] < \infty$. Furthermore, under these assumptions, the distribution of $R$ is the unique solution with finite $\beta$-moment to recursion \eqref{eq:Linear}. 
\end{lem}

\begin{proof}
Since $R^{(n )}\to R$ a.s., the result will follow from Slutsky's Theorem (see Theorem 25.4, p. 332 in \cite{Billingsley_1995}) once we show that $W_n(R_0^*) \Rightarrow 0$. To this end, note that $W_n(R_0^*)$, as defined by \eqref{eq:LastWeights}, is the same as $W_n$ if we substitute the $Q_{{\bf i}}$ by the $R_{0,{\bf i}}^*$. Then, for every $\epsilon > 0$ we have that
\begin{align*}
P( W_n(R_0^*) > \epsilon) &\leq \epsilon^{-\beta} E[ W_n(R_0^*)^\beta] \\
&\leq \epsilon^{-\beta} \rho_\beta^n E[(R_0^*)^\beta] \qquad \text{(by Lemma \ref{L.MomentSmaller_1})} .
\end{align*}
Since by assumption the right-hand side converges to zero as $n \to \infty$, then $R_n^* \Rightarrow R$. Furthermore, $E[R^\beta] < \infty$ by Lemma \ref{L.Moments_R}. Clearly, under the assumptions, the distribution of $R$ represents the unique solution to \eqref{eq:Linear}, since any other possible solution with finite $\beta$-moment would have to converge to the same limit.  
\end{proof}

{\sc Remarks:} (i) Note that when $E[N] < 1$ the branching tree is a.s. finite and no conditions on the $\{C_i\}$ are necessary for $R < \infty$ a.s. This corresponds to the second condition in Theorem 1 of \cite{Brandt_86}. (ii) In view of the same theorem from \cite{Brandt_86}, one could possibly establish the convergence of $R_n^* \Rightarrow R < \infty$ under milder conditions. However, since in this paper we only study the power tails of $R$, the assumptions of Lemma \ref{L.Convergence} are not restrictive.  (iii) Note that if $E\left[ \sum_{i=1}^N C_i^\alpha \right] = 1$ with $\alpha \in (0, 1]$, then there might not be a $0 < \beta < \alpha$ for which $E\left[ \sum_{i=1}^N C_i^\beta \right] < 1$, e.g., the case of deterministic $C_i$'s that was studied in \cite{Rosler_93}.

%%% Main Result for Linear Recursion
\subsection{Main result} \label{SS.MainLinear}

We now characterize the tail behavior of the distribution of the solution $R$ to the nonhomogeneous equation \eqref{eq:Linear}, 
as defined by \eqref{eq:ExplicitConstr}.

\begin{thm} \label{T.LinearRecursion}
Let $(Q, N, C_1, C_2, \dots)$ be a nonnegative random vector, with $N \in \mathbb{N} \cup \{\infty\}$, $P(Q > 0) > 0$
and $R$ be the solution to \eqref{eq:Linear} given by  \eqref{eq:ExplicitConstr}. 
Suppose that there exists $j \geq 1$ with $P(N\ge j,C_j>0)>0$ such that the measure $P\left(\log C_j\in du, C_j > 0, N\ge j\right)$ is nonarithmetic, and 
that for some $\alpha > 0$,  $E[Q^\alpha] < \infty$,  $0 < E \left[ \sum_{i=1}^N C_i^\alpha \log C_i \right] < \infty$ and  $ E \left[ \sum_{i=1}^N C_i^\alpha \right] = 1$. In addition, assume
\begin{enumerate}
\item $E\left[ \sum_{i=1}^N C_i \right] < 1$ and $E\left[ \left( \sum_{i=1}^N C_i \right)^\alpha \right] < \infty$, if $\alpha > 1$; or,
\item $E\left[ \left( \sum_{i=1}^N C_i^{\alpha/(1+\epsilon)}\right)^{1+\epsilon} \right] < \infty$ for some $0 < \epsilon< 1$, if $0 < \alpha \leq 1$. 
\end{enumerate}
Then,
$$P(R > t) \sim H t^{-\alpha}, \qquad t \to \infty,$$
where $0 \leq H < \infty$ is given by
\begin{align*}
H &= \frac{1}{E\left[ \sum_{i=1}^N C_i^\alpha \log C_i  \right] } \int_{0}^\infty v^{\alpha-1} \left( P(R > v) - E\left[ \sum_{i=1}^{N} \Indicator(C_{i} R > v ) \right]    \right) dv \\
&= \frac{E\left[ \left( \sum_{i=1}^N C_i R_i +Q \right)^\alpha - \sum_{i=1}^N (C_i R_i )^\alpha \right]}{\alpha E\left[ \sum_{i=1}^N C_i^\alpha \log C_i  \right] }.
\end{align*}
\end{thm}

{\sc Remarks:} (i) The nonhomogeneous equation has been previously studied for the special case when $Q$ and the $\{C_i\}$ are deterministic constants. In particular, Theorem 5 of \cite{Rosler_93} analyzes the solutions to \eqref{eq:Linear} when $Q$ and the $\{C_i\}$ are nonnegative deterministic constants, 
which, when $\sum_{i=1}^N C_i^\alpha =1$, $\alpha>0$, implies that $C_i \leq 1$ for all $i$ and $\sum_{i} C_i^\alpha \log C_i \leq 0$, falling outside of the scope of this paper. The solutions to \eqref{eq:Linear} for the case when $Q$ and the $C_i$'s are real valued deterministic constants were analyzed in \cite{Alsm_Rosl_05}. For the very recent work (published on arXiv after the first draft of this paper) that characterizes all the solutions to \eqref{eq:Linear} for $Q$ and $\{C_i\}$ random see \cite{Alsm_Mein_10}. (ii) When $\alpha > 1$, the condition $E\left[ \left( \sum_{i=1}^N C_i \right)^\alpha \right] < \infty$ is needed to ensure that the tail of $R$ is not dominated by $N$. In particular, if the $\{C_i\}$ are iid and independent of $N$, the condition reduces to $E[N^\alpha] < \infty$ since $E[C^\alpha] < \infty$ is implied by the other conditions; see Theorems 4.2 and 5.4 in \cite{Jel_Olv_09}. Furthermore, when $0 < \alpha \leq 1$ the condition $E\left[ \left( \sum_{i=1}^N C_i \right)^\alpha \right] < \infty$ is redundant since $E\left[ \left( \sum_{i=1}^N C_i \right)^\alpha \right] \leq E\left[  \sum_{i=1}^N C_i^\alpha \right] = 1$, and the additional condition $E\left[ \left( \sum_{i=1}^N C_i^{\alpha/(1+\epsilon)} \right)^{1+\epsilon} \right] < \infty$ is needed. When the $\{C_i \}$ are iid and independent of $N$, the latter condition reduces to $E[N^{1+\epsilon}] < \infty$ (given the other assumptions), which is consistent with Theorem 4.2 in \cite{Jel_Olv_09}. 
(iii) Note that the second expression for $H$ is more suitable for actually computing it, especially in the case of $\alpha$ being an integer, as will be stated in the forthcoming Corollary~\ref{C.explicit}.  When $\alpha>1$ is not an integer, we can derive an explicit upper bound on $H$ by using Lemma~\ref{L.Max_Approx}. 
Regarding the lower bound, the elementary inequality 
$\left( \sum_{i=1}^k x_i \right)^\alpha \ge \sum_{i=1}^k x_i^\alpha$ for $\alpha\ge1$ and $x_i \geq 0$, yields
$$
H\ge \frac{E\left[ Q^\alpha  \right]}{\alpha E\left[ \sum_{i=1}^N C_i^\alpha \log C_i  \right] }>0.
$$
Similarly, for $0<\alpha<1$, using the corresponding inequality 
$\left( \sum_{i=1}^k x_i \right)^\alpha \le \sum_{i=1}^k x_i^\alpha$ for $0<\alpha\le1$, $x_i \geq 0$, we obtain
$H\le {E\left[ Q^\alpha  \right]}/{\left(\alpha E\left[ \sum_{i=1}^N C_i^\alpha \log C_i  \right] \right)}.$
(iv) Let us also observe that the solution $R$, given by  \eqref{eq:ExplicitConstr}, to equation \eqref{eq:Linear} 
may be a constant (non power law) $R=r>0$ when 
$P(r=Q+r \sum_{i=1}^N C_i)=1$. However, similarly as in remark (i), such a solution is excluded from the theorem since 
$P(r=Q+r\sum_{i=1}^N C_i)=1$ implies $E[\sum_i C_i^\alpha \log C_i]\le 0, \alpha>0$.

Before proceeding with the proof of Theorem \ref{T.LinearRecursion}, we need the following two technical results; their proofs are given in Section \ref{SS.LinearProofs}. Lemma \ref{L.Max_Approx} below will also be used in subsequent sections for other recursions. With some abuse of notation, we will use throughout the paper $\max_{1 \leq i \leq N} x_i$ to denote $\sup_{1 \leq i < N+1} x_i$ in case $N = \infty$.

\begin{lem} \label{L.Max_Approx}
Suppose $(N, C_1, C_2, \dots)$ is a nonnegative random vector, with $N \in \mathbb{N} \cup \{\infty\}$ and let $\{R_i\}_{i \in \mathbb{N}}$ be a sequence of iid nonnegative random variables independent of $(N, C_1, C_2, \dots)$ having the same distribution as $R$. For $\alpha >0$, suppose that $\sum_{i=1}^N (C_i R_i)^\alpha < \infty$ a.s. and $E[R^\beta]< \infty$ for any $0 < \beta < \alpha$.  Furthermore, assume that  $E\left[ \left( \sum_{i=1}^N C_i^{\alpha/(1+\epsilon)}\right)^{1+\epsilon} \right] < \infty$ for some $0 < \epsilon< 1$. Then, 
\begin{align*}
0 &\leq \int_{0}^\infty \left( E\left[ \sum_{i=1}^N \Indicator(C_i R_i > t ) \right] - P\left( \max_{1\leq i \leq N} C_i R_i > t \right) \right)  t^{\alpha -1} \, dt \\
&= \frac{1}{\alpha} E \left[  \sum_{i=1}^N  \left(C_i R_i \right)^\alpha - \left( \max_{1\leq i \leq N} C_i R_i \right)^\alpha   \right]  < \infty.
\end{align*}
\end{lem}

\begin{lem} \label{L.ExtraQ}
Let $(Q, N, C_1, C_2, \dots)$ be a nonnegative vector with $N \in \mathbb{N} \cup \{\infty\}$ and let $\{R_i\}$ be a sequence of iid random variables, independent of $(Q, N, C_1, C_2, \dots)$. Suppose that for some $\alpha > 1$ we have $E[Q^\alpha] < \infty$, $E\left[ \left( \sum_{i=1}^N C_i \right)^\alpha \right] < \infty$,  $E[R^\beta] < \infty$ for any $0 < \beta < \alpha$, and $\sum_{i=1}^N C_i R_i < \infty$ a.s. Then
$$E \left[ \left( \sum_{i=1}^N C_i R_i + Q  \right)^\alpha - \sum_{i=1}^N \left(  C_i R_i \right)^\alpha   \right]  < \infty.$$
\end{lem}

\bigskip

\begin{proof}[Proof of Theorem \ref{T.LinearRecursion}]
By Lemma \ref{L.Moments_R}, we know that $E[R^\beta] < \infty$ for any $0 < \beta < \alpha$. To verify that $E\left[ \sum_{i=1}^N C_i^\gamma \right] < \infty$ for some $0 \leq \gamma < \alpha$ note that if $\alpha > 1$ we have, by the assumptions of the theorem and Jensen's inequality, 
$$E\left[ \sum_{i=1}^N C_i^\gamma \right] \leq E\left[ \left( \sum_{i=1}^N C_i \right)^\gamma \right] \leq \left( E\left[ \left(\sum_{i=1}^N C_i \right)^\alpha \right] \right)^{\gamma/\alpha} < \infty$$
for any $1 \leq \gamma < \alpha$. If $0 < \alpha \leq 1$, then for $\gamma = \alpha (1+\epsilon/2)/(1+\epsilon) < \alpha$ we have
$$E\left[ \sum_{i=1}^N C_i^\gamma \right] \leq E\left[ \left( \sum_{i=1}^N C_i^{\alpha/(1+\epsilon)} \right)^{1+\epsilon/2} \right] \leq \left( E\left[ \left( \sum_{i=1}^N C_i^{\alpha/(1+\epsilon)} \right)^{1+\epsilon} \right]  \right)^\frac{1+\epsilon/2}{1+\epsilon} < \infty.$$

The statement of the theorem with the first expression for $H$ will follow from Theorem \ref{T.NewGoldie} once we prove that condition \eqref{eq:Goldie_condition} holds.  To this end, define
$$R^* = \sum_{i=1}^N C_i R_i + Q.$$
Then, 
\begin{align*}
\left| P(R>t) - E\left[ \sum_{i=1}^N \Indicator(C_iR_i > t ) \right] \right|  &\leq \left| P(R > t) - P\left( \max_{1\leq i \leq N} C_i R_i > t \right) \right|   \\
&\hspace{5mm} + \left| P\left( \max_{1\leq i \leq N} C_i R_i > t  \right)   - E\left[ \sum_{i=1}^N \Indicator(C_iR_i > t ) \right]  \right|.
\end{align*}
Since $R \stackrel{\mathcal{D}}{=} R^* \geq \max_{1\leq i\leq N} C_i R_i$, the first absolute value disappears. For the second one, note that 
\begin{align*}
&E\left[ \sum_{i=1}^N \Indicator(C_iR_i > t ) \right]  -  P\left(  \max_{1\leq i \leq N} C_i R_i > t  \right)  \\
&= E\left[ \sum_{i=1}^N \Indicator(C_iR_i > t ) \right]  - E\left[ \Indicator\left(  \max_{1\leq i \leq N} C_i R_i > t   \right) \right] \geq 0.
\end{align*}
Now it follows that
\begin{align}
&\left| P(R > t) - E\left[ \sum_{i=1}^N \Indicator(C_iR_i > t ) \right] \right|  \notag \\
&\leq P(R > t) - P\left( \max_{1\leq i \leq N} C_i R_i > t \right) \notag \\
&\hspace{5mm} + E\left[ \sum_{i=1}^N \Indicator(C_iR_i > t ) \right] - P\left( \max_{1\leq i \leq N} C_i R_i > t \right)  . \label{eq:Term2}
\end{align}

Note that the integral corresponding to \eqref{eq:Term2} is finite by Lemma \ref{L.Max_Approx} if we show that the assumptions of Lemma~\ref{L.Max_Approx} are satisfied when $\alpha > 1$. Note that in this case we can choose $\epsilon > 0$ such that $\alpha/(1+\epsilon) \geq 1$ and use the inequality 
\begin{equation} \label{eq:concaveSum}
\sum_{i=1}^k x_i^\beta \le \left( \sum_{i=1}^k x_i \right)^\beta  
\end{equation}
for $\beta \ge 1$, $x_i \geq 0$, $k \leq \infty$ to obtain
$$E\left[ \left( \sum_{i=1}^N C_i^{\alpha/(1+\epsilon)} \right)^{1+\epsilon} \right] \leq E\left[ \left( \sum_{i=1}^N C_i \right)^\alpha \right]  < \infty.$$
Therefore, it only remains to show that
\begin{equation} \label{eq:RecVsMax}
\int_0^\infty  \left( P(R > t) - P\left( \max_{1\leq i \leq N} C_i R_i > t \right) \right) t^{\alpha-1} \, dt  < \infty.
\end{equation}

To see this, note that $R \stackrel{\mathcal{D}}{=} R^*$ and $1(R^* > t) - 1(\max_{1\leq i\leq N} C_i R_i > t) \geq 0$, and thus, by Fubini's theorem, we have
\begin{equation*}
\int_0^\infty  \left( P(R > t) - P\left( \max_{1\leq i \leq N} C_i R_i > t \right) \right) t^{\alpha-1} \, dt = \frac{1}{\alpha} E \left[ (R^*)^\alpha - \left( \max_{1\leq i \leq N} C_i R_i \right)^\alpha   \right].
\end{equation*}

If $0 < \alpha \leq 1$, we apply \eqref{eq:concaveSum} to obtain
$$E \left[ (R^*)^\alpha - \left( \max_{1\leq i \leq N} C_i R_i \right)^\alpha   \right] \leq E \left[ Q^\alpha + \sum_{i=1}^N (C_iR_i)^\alpha - \left( \max_{1\leq i \leq N} C_i R_i \right)^\alpha   \right],$$ 
which is finite by Lemma \ref{L.Max_Approx} and the assumption $E[Q^\alpha] < \infty$. 

If $\alpha > 1$, we have $\left(\sum_{i=1}^k x_i \right)^\alpha \geq \sum_{i=1}^k x_i^\alpha$, $x_i \geq 0$, $k \leq \infty$, implying that we can split the expectation as follows
\begin{align*}
E \left[ (R^*)^\alpha - \left( \max_{1\leq i \leq N} C_i R_i \right)^\alpha   \right]  &=   E \left[ (R^*)^\alpha - \sum_{i=1}^N  \left(C_i R_i \right)^\alpha   \right]  \\
&\hspace{5mm} +   E \left[  \sum_{i=1}^N  \left(C_i R_i \right)^\alpha - \left( \max_{1\leq i \leq N} C_i R_i \right)^\alpha   \right],
\end{align*}
which can be done since both expressions inside the expectations on the right-hand side are nonnegative. The first expectation is finite by Lemma \ref{L.ExtraQ} and the second expectation is again finite by Lemma \ref{L.Max_Approx}. 

Finally, applying Theorem \ref{T.NewGoldie} gives 
$$P(R > t) \sim H t^{-\alpha},$$
where $H = \left( E\left[ \sum_{j=1}^N C_j^\alpha \log C_j  \right] \right)^{-1} \int_{0}^\infty v^{\alpha-1} \left( P(R > v) - E\left[ \sum_{j=1}^{N} \Indicator(C_{j} R > v ) \right]    \right) dv$. 

To obtain the second expression for $H$ note that
\begin{align}
&\int_{0}^\infty v^{\alpha-1} \left( P(R > v) - E\left[ \sum_{j=1}^{N} \Indicator(C_{j} R > v) \right]    \right) dv \notag \\
&= \int_0^\infty v^{\alpha-1}   E\left[\Indicator\left(\sum_{i=1}^N C_i R_i + Q > v \right) - \sum_{i=1}^N \Indicator(C_iR_i > v)  \right] \, dv \notag \\ 
&= E \left[   \int_0^\infty v^{\alpha-1}  \left(  \Indicator\left(\sum_{i=1}^N C_i R_i + Q > v \right) - \sum_{i=1}^N \Indicator(C_iR_i > v) \right) dv  \right] \label{eq:Fubini} \\
&= E \left[   \int_0^{\sum_{i=1}^N C_i R_i + Q} v^{\alpha-1} dv  -  \sum_{i=1}^N \int_0^{C_i R_i} v^{\alpha-1}  dv  \right] \label{eq:DiffIntegrals} \\
&= \frac{1}{\alpha} E\left[ \left( \sum_{i=1}^N C_i R_i + Q \right)^\alpha - \sum_{i=1}^N (C_i R_i)^\alpha   \right] , \notag
\end{align}
where \eqref{eq:Fubini} is justified by Fubini's Theorem and the integrability of 
\begin{align*}
&v^{\alpha-1} \left| \Indicator\left(\sum_{i=1}^N C_i R_i + Q > v \right) - \sum_{i=1}^N \Indicator(C_iR_i > v) \right| \\
&\leq v^{\alpha-1} \left( \Indicator\left(\sum_{i=1}^N C_i R_i + Q > v \right) - \Indicator\left( \max_{1\leq i \leq N} C_i R_i > v \right) \right) \\
&\hspace{5mm} + v^{\alpha-1} \left( \sum_{i=1}^N \Indicator(C_iR_i > v) - \Indicator\left( \max_{1\leq i \leq N} C_i R_i > v \right) \right),
\end{align*}
which is a consequence of \eqref{eq:RecVsMax} and Lemma \ref{L.Max_Approx}; and \eqref{eq:DiffIntegrals} follows from the observation that 
$$v^{\alpha-1} \Indicator\left(\sum_{i=1}^N C_i R_i + Q > v\right) \qquad \text{and} \qquad v^{\alpha-1} \sum_{i=1}^N \Indicator(C_iR_i > v)$$
are each almost surely absolutely integrable with respect to $v$ as well. 

This completes the proof.
\end{proof}

As indicated earlier, when $\alpha\ge 1$ is an integer, we can obtain the following explicit expression for $H$.

\begin{cor} \label{C.explicit}
For integer $\alpha \geq 1$, and under the same assumptions of Theorem \ref{T.LinearRecursion}, the constant $H$ can be explicitly computed as a function of $E[R^k], E[C^k], E[Q^k]$, $0 \leq k \leq \alpha-1$. In particular, for $\alpha = 1$,
$$H = \frac{E[Q]}{E\left[ \sum_{i=1}^N C_i \log C_i  \right] },$$
and for $\alpha = 2$,
\begin{align*}
H &= \frac{E[Q^2] + 2 E[R] E\left[ Q \sum_{i=1}^N C_i \right] + 2 (E[R])^2 E\left[ \sum_{i=1}^N \sum_{j=i+1}^N C_i C_j  \right] }{2 E\left[ \sum_{i=1}^N C_i^2 \log C_i  \right] },\\
&E[R] = \frac{E[Q]}{1-E\left[ \sum_{i=1}^N C_i \right]}.
\end{align*}
\end{cor}

\begin{proof}
The proof follows directly from multinomial expansions of the second expression for $H$ in Theorem~\ref{T.LinearRecursion}.
\end{proof}

%%% Homogeneous linear recursion
\subsection{The homogeneous recursion} \label{SS.Homogeneous}

In this section we briefly describe how the methodology developed in the previous sections can 
be applied to study the critical, $E\left[ \sum_{i=1}^N C_i \right] = 1$, homogeneous linear recursion
\begin{equation} \label{eq:LinearHomogeneous}
R \stackrel{\mathcal{D}}{=} \sum_{i=1}^N C_i R_i,
\end{equation}
where $(N, C_1, C_2, \dots)$ is a nonnegative random vector with $N \in \mathbb{N} \cup \{\infty\}$ and $\{ R_i\}_{i \in \mathbb{N}}$ is a sequence of iid random variables independent of $(N, C_1, C_2, \dots)$ having the same distribution as $R$. This equation has been studied extensively in the literature under various different assumptions; for recent results see
 \cite{Liu_00, Iksanov_04, Alsm_Kuhl_07} and the references therein. 
 
Based on the model from Section~\ref{S.LinearRec} we can construct a solution to \eqref{eq:LinearHomogeneous} as follows. 
Consider the process $\{W_n\}_{n\ge 0}$ defined by \eqref{eq:W_k} with $Q_{\bf i}\equiv 1$. Then,  the $\{W_n\}$ satisfy in distribution 
the homogeneous recursion in \eqref{eq:WnRec} and, given that $E\left[ \sum_{i=1}^N C_i \right] = 1$,
we have $E[W_n]=1$. Hence, $\{W_n\}_{n\ge 0}$ is a nonnegative martingale and by the martingale convergence theorem $W_n\to R$ a.s. with $E[R]\leq 1$.
Next, provided that  
$$
E\left[ \sum_{i=1}^N C_i \log C_i \right] < 0 \quad\text{ and }\quad 
E\left[ \left( \sum_{i=1}^N C_i \right) \log^+ \left( \sum_{i=1}^N C_i \right) \right] < \infty$$ 
it can be shown that $E[R]=1$, see Theorem 1.1(d) in \cite{Alsm_Kuhl_07} (see also Theorem 2 in \cite{Liu_00});
$\log^+x=\max(\log x,0)$. 
Furthermore, as argued in equation (1.9) of \cite{Alsm_Kuhl_07}, it can easily be shown  that this $R$ is a solution to \eqref{eq:LinearHomogeneous}.
Note that the same construction of the solution $R$ on a branching tree was given in \cite{Alsm_Kuhl_07} and \cite{Liu_00}.
Since the solutions to \eqref{eq:LinearHomogeneous} are scale invariant, this construction also shows that for any $m>0$ there is a solution $R$ 
with mean $m$; or equivalently, it is enough to study the solutions with mean $1$. 
Moreover, under additional assumptions it can be shown that this constructed $R$ is the only solution with mean $1$, e.g. see \cite{Liu_98,Liu_00,Iksanov_04}.
However, it is not the objective of this section to study the uniqueness of this solution, rather we focus on studying the tail behavior of any such possible solution
(since our Theorem~~\ref{T.NewGoldie} does not require the uniqueness of $R$).
As a side note, we point out that  \eqref{eq:LinearHomogeneous} can have solutions if $E\left[ \sum_{i=1}^N C_i^\beta \right]=1$ 
for some $0<\beta<1$, as studied in \cite{Liu_98,Iksanov_04}.

A version of the following theorem, with a possibly less explicit constant, was previously proved in Theorem~2.2 in \cite{Liu_00} 
and Proposition~7 in \cite{Iksanov_04}; they also study the lattice case.  Regarding the lattice case, as pointed out earlier in the remark after 
Theorem~\ref{T.NewGoldie}, all the results in this paper can be developed for this case as well by using the corresponding renewal theorem. 

\begin{thm} \label{T.LinearHomog}
Suppose that there exists $j \geq 1$ with $P(N\ge j,C_j>0)>0$ such that the measure $P(\log C_j\in du, C_j > 0, N\ge j)$ is nonarithmetic.
Suppose further that for some $\alpha > 1$, $E\left[ \left( \sum_{i=1}^N C_i \right)^\alpha \right] < \infty$, $E \left[ \sum_{i=1}^N C_i^\alpha \log^+ C_i \right] < \infty$ and  $E\left[ \sum_{i=1}^N C_i \right] = E \left[ \sum_{i=1}^N C_i^\alpha \right] = 1$. Then, equation \eqref{eq:LinearHomogeneous} has a solution $R$ with $0<E[R] <\infty$ such that
$$P(R > t) \sim H t^{-\alpha}, \qquad t \to \infty,$$
where $0 \leq H < \infty$ is given by
\begin{align*}
H &= \frac{1}{E\left[ \sum_{i=1}^N C_i^\alpha \log C_i  \right] } \int_{0}^\infty v^{\alpha-1} \left( P(R > v) - E\left[ \sum_{i=1}^{N} \Indicator(C_{i} R > v ) \right]    \right) dv \\
&= \frac{E\left[ \left( \sum_{i=1}^N C_i R_i \right)^\alpha - \sum_{i=1}^N (C_i R_i )^\alpha \right]}{\alpha E\left[ \sum_{i=1}^N C_i^\alpha \log C_i  \right] }.
\end{align*}
Furthermore, if $P(\tilde{N} \geq 2) >0$, $\tilde{N}= \sum_{i=1}^N 1(C_i>0)$, then $H>0$.
\end{thm}

\begin{proof} 
By the assumptions, the function $\varphi(\theta) \triangleq E\left[ \sum_{i=1}^N C_j^\theta \right]$ is convex, finite, and continuous on $[1, \alpha]$, since $\varphi(1) = \varphi(\alpha) = 1$. 
Furthermore, by standard arguments, it can be shown that both $\varphi'(\theta)$ and $\varphi''(\theta)$ exist on the open interval $(1, \alpha)$ and, in particular,
$$\varphi''(\theta) = E\left[ \sum_{i=1}^N C_i^\theta (\log C_i)^2   \right].$$
Clearly, $\varphi''(\theta) > 0$ provided that $P( C_i \in \{0,1\},1 \leq  i \leq N) < 1$. To see that this is indeed the case, note that $E\left[ \sum_{i=1}^N C_i \right] = 1$ implies that $P(C_i \equiv 0, 1 \leq i \leq N) < 1$, which combined with the nonarithmetic assumption yields $P( C_i \in \{0,1\},1 \leq  i \leq N) < 1$.  Hence, there exists $1 < \theta_1 < \theta_2 < \alpha$ such that $\varphi'(\theta_1) < 0$ and $\varphi'(\theta_2) > 0$, 
implying by the monotonicity of $\varphi'(\cdot)$ and monotone convergence that
\begin{equation} \label{eq:Derivative_alpha}
0 < \varphi'(\alpha-) = E\left[ \sum_{i=1}^N C_i^\alpha \log C_i \right]  \leq E\left[ \sum_{i=1}^N C_i^{\alpha} \log^+ C_i \right] < \infty \qquad \text{and}  
\end{equation}
\begin{equation*} 
 \varphi'(1+) = E\left[ \sum_{i=1}^N C_i \log C_i \right] < 0 .
\end{equation*}
The last expression and the observation $E\left[ \left( \sum_{i=1}^N C_i \right) \log^+ \left( \sum_{i=1}^N C_i \right) \right] < \infty$ (implied by $E\left[ \left( \sum_{i=1}^N C_i \right)^\alpha \right] < \infty$) yields, as argued at the beginning of this section, that recursion \eqref{eq:LinearHomogeneous} has a solution with finite positive mean, see Theorem 1.1(d) and equation (1.9) in \cite{Alsm_Kuhl_07} (see also Theorem 2 in \cite{Liu_00}). 

Next, in order to apply Theorem \ref{T.NewGoldie}, we use \eqref{eq:Derivative_alpha} and $E[ R^\beta] < \infty$ for any $0 < \beta < \alpha$; the latter follows from Theorem 3.1 in \cite{Alsm_Kuhl_07} and the strict convexity of $\varphi(\cdot)$ argued above (see also, Proposition~4 in \cite{Iksanov_04} and Theorem~2.1 in \cite{Liu_00}). The rest of the proof, except for the $H>0$ part, proceeds exactly as that of Theorem~\ref{T.LinearRecursion} by setting $Q \equiv 0$. 

Regarding the $H>0$ statement, the assumption $P(\tilde{N} \geq 2) >0$ implies that there exist $1 \leq n \leq \infty$ and $1\le i_1< i_2 < n+1$ such that $P(N = n, C_{i_1}>0,C_{i_2}>0)>0$, which further yields, for some $\delta>0$,
\begin{equation}
\label{eq:N2}
P(N\ge i_2, C_{i_1}>\delta,C_{i_2}>\delta)>0.
\end{equation}
Next, by using the inequality $\left(  x_1 + x_2 \right)^\alpha \geq x_1^\alpha + x_2^\alpha$ for $x_1, x_2 \geq 0$ and $\alpha > 1$, 
the second expressions for $H$ in the theorem can be bounded from below by 
\begin{equation}
\label{eq:Hlb1}
H\ge \frac{E\left[ 1(N\ge i_2) \left(\left(C_{i_1} R_{i_1} + C_{i_2} R_{i_2}\right)^\alpha - (C_{i_1} R_{i_1})^\alpha - (C_{i_2} R_{i_2})^\alpha\right) \right]}{\alpha E\left[ \sum_{i=1}^N C_i^\alpha \log C_i  \right] }.
\end{equation}
To further bound the numerator in \eqref{eq:Hlb1} we define the function
$$
f(x)=(1+x)^\alpha-1-x^\alpha- c x^{\alpha-\epsilon},
$$
where $0<\epsilon<\alpha-1$, $0<c<2^\gamma-1$ and $\gamma=\alpha-1-\epsilon$. It can be shown that $f(x) \geq 0$ for $x \in [0,1]$, since $f(0)=0$ and $f'(x)\ge \alpha x^\gamma ((1+1/x)^\gamma-1-c)\ge 0$ on $[0,1]$.
Hence, by using the inequality $f(x) \geq 0$, we derive for $x_1\ge 0, x_2\ge 0$, $\max\{x_1,x_2\}>0$
and $x={\min\{x_1,x_2\}}/{\max\{x_1,x_2\}}$
\begin{align*}
(x_1+x_2)^\alpha-x_1^\alpha-x_2^\alpha &= (\max\{x_1,x_2\})^\alpha \left((1+x)^\alpha    - 1 -  x^\alpha \right)
\\
&\geq c (\max\{x_1,x_2\})^\alpha x^{\alpha-\epsilon}\ge c (\min\{x_1,x_2\})^{\alpha};
\end{align*}
the inequality clearly holds even if $\max\{x_1,x_2\}=0$ since both of its sides are zero.
Thus, by applying this last inequality to \eqref{eq:Hlb1} and using \eqref{eq:N2}, we obtain 
\begin{align*}
H &\geq \frac{c E\left[ 1(N\ge i_2) \left(\min \left\{C_{i_1} R_{i_1} , C_{i_2} R_{i_2}\right\}\right)^\alpha \right]}{\alpha E\left[ \sum_{i=1}^N C_i^\alpha \log C_i  \right] } \\
&\geq  \frac{c \delta^\alpha P(N\ge i_2, C_{i_1}>\delta, C_{i_2}>\delta) E[\left(\min \{R_{i_1},R_{i_2}\}\right)^\alpha] }{\alpha E\left[ \sum_{i=1}^N C_i^\alpha \log C_i  \right] }>0.
\end{align*}
This completes the proof.
\end{proof}

{\sc Remarks:} (i) Note that the assumptions of Theorem \ref{T.LinearHomog} differ slightly from those of Theorem \ref{T.LinearRecursion} in the condition $0 < E\left[ \sum_{i=1}^N C_i^\alpha \log C_i \right] <~\infty$, which is implied by $E\left[ \sum_{i=1}^N C_i^\alpha \log^+ C_i \right] < \infty$, the strict convexity of $\varphi(\theta) = E\left[ \sum_{i=1}^N C_i^\theta \right]$ and the hypothesis that $\varphi(1) = \varphi(\alpha) = 1$, as argued in the preceding proof. 
(ii) The assumption $P(\tilde{N} \geq 2) >0$ is the minimal one to ensure the existence of a nontrivial solution, 
see conditions (H0) in \cite{Liu_98} and (C4) in  \cite{Alsm_Kuhl_07}.  Otherwise, when $P(\tilde{N} \le 1) =1$, $W_n$ is a simple
multiplicative random walk with no branching; clearly, in this case our expression for $H$ reduces to zero. 
Also, if $P(\sum_{i=1}^N C_i=1)=1$, $R$ can only be a constant; see the remark above 
Theorem~2.1 in \cite{Liu_00}. 
However, this last case is excluded from the theorem since $P(\sum_{i=1}^N C_i=1)=1$ implies $C_i\le 1$ a.s., 
which, in conjunction with $\varphi(\alpha) = 1, \alpha>1$, yields $P( C_i \in \{0,1\},1 \leq  i \leq N) = 1$, 
but this cannot happen due to the nonarithmetic assumption. 
(iii) Note also that condition (C3) in  \cite{Alsm_Kuhl_07} (equivalent to $P( C_i \in \{0,1\},1 \leq  i \leq N) < 1$ in our notation) 
is implied by the nonarithmetic assumption of our theorem. 
Interestingly enough, if this last condition fails, Lemma~1.1 of \cite{Liu_98} shows that equation \eqref{eq:LinearHomogeneous} has no nontrivial solutions.
(iv) As stated earlier, a version of this theorem was proved in Theorem 2.2 of \cite{Liu_00}, by transforming recursion \eqref{eq:LinearHomogeneous} into a first order difference (autoregressive/perpetuity) equation on a different probability space, see Lemma 4.1 in \cite{Liu_00}. However,  it appears that the method from \cite{Liu_00} does not extend to the nonhomogeneous and non-linear problems that we cover here, since the proof of Lemma~4.1 in \cite{Liu_00} critically depends on having both $E[R] = 1$ and $E\left[ \sum_{i=1}^N C_i \right] = 1$.

Similarly as in Corollary \ref{C.explicit}, the constant $H$ can be computed explicitly for integer $\alpha \geq 2$.

\begin{cor} 
\label{C.explicitHom}
For integer $\alpha \geq 2$, and under the same assumptions of Theorem \ref{T.LinearHomog}, the constant $H$ can be explicitly computed as a function of $E[R^k], E[C^k]$, $1 \leq k \leq \alpha-1$. In particular, for $\alpha = 2$,
$$H = \frac{ E\left[ \sum_{i=1}^N \sum_{j=i+1}^N C_i C_j  \right] }{E\left[ \sum_{i=1}^N C_i^2 \log C_i  \right] }.$$
\end{cor}

\begin{proof}
The proof follows directly from multinomial expansions of the second expression for $H$ in Theorem~\ref{T.LinearHomog}. 
\end{proof}

We also want to point out that for non-integer general $\alpha > 1$ we can use Lemma \ref{L.Alpha_Moments} to obtain the following bound for $H$,
$$H \leq \frac{ \left( E\left[ R^{p-1} \right] \right)^{\alpha/(p-1)} E\left[ \left( \sum_{i=1}^N C_i  \right)^\alpha\right]}{\alpha E\left[ \sum_{i=1}^N C_i^\alpha \log C_i  \right] },$$
where $p = \lceil \alpha \rceil$.

%%%%%%%%%%%%% The maximum recursion
\section{The maximum recursion: $R = \left( \bigvee_{i=1}^N C_i R_i \right) \vee Q$} \label{S.MaxRec}

In order to show the general applicability of the implicit renewal theorem,  we study in this section the following non-linear recursion:
\begin{equation} \label{eq:Maximum}
R \stackrel{\mathcal{D}}{=} \left(\bigvee_{i=1}^N C_i R_i \right) \vee Q,
\end{equation}
where $(Q, N, C_1, C_2, \dots)$ is a nonnegative random vector with $N \in \mathbb{N} \cup \{\infty\}$, \linebreak $P(Q > 0) > 0$ and $\{R_i\}_{i\in \mathbb{N}}$ is a sequence of iid random variables independent of $(Q, N, C_1, C_2, \dots)$ having the same distribution as $R$. Note that in the case of page ranking applications, where the $\{R_i\}$ represent the ranks of the neighboring pages, the potential ranking algorithm defined by the preceding recursion, determines the rank of a page as a weighted version of the most highly ranked neighboring page. In other words, the highest ranked reference has the dominant impact. Similarly to the homogeneous linear case, this recursion was previously studied in \cite{Alsm_Rosl_08} under the assumption that $Q \equiv 0$, $N = \infty$, and the $\{C_i\}$ are real valued deterministic constants. The more closely related case of $Q \equiv 0$ and $\{C_i \} \geq 0$ being random was studied earlier in \cite{Jag_Ros_04}. Furthermore, these max-type stochastic recursions appear in a wide variety of applications, ranging from the average case analysis of algorithms to statistical physics; see \cite{Aldo_Band_05} for a recent survey.  

Using standard arguments, we start by constructing a solution to \eqref{eq:Maximum} on a tree and then we show that this solution is finite a.s. and unique under iterations and some moment conditions, similar to what was done for the nonhomogeneous linear recursion in Section \ref{S.LinearRec}. Our main result of this section is stated in Theorem \ref{T.MaximumRecursion}. 

Following the same notation as in Section \ref{S.LinearRec}, define the process
\begin{equation} \label{eq:V_k}
V_n = \bigvee_{{\bf i} \in A_n} Q_{{\bf i}} \boldC_{{\bf i}}, \qquad n \geq 0,
\end{equation}
on the weighted branching tree $\mathcal{T}_{Q, C}$, as constructed in Section \ref{S.ModelDescription}.
Recall that the convention is that  $(Q, N, C_1, C_2, \dots) = (Q_\emptyset, N_\emptyset, C_{(\emptyset, 1)}, C_{(\emptyset, 2)}, \dots)$ denotes the random vector corresponding to the root node. 

With a possible abuse of notation relative to Section \ref{S.LinearRec}, define the process $\{R^{(n)}\}_{n \geq 0}$ according to
$$R^{(n)} = \bigvee _{k=0}^n V_k, \qquad n \geq 0.$$
Just as with the linear recursion from Section \ref{S.LinearRec}, it is not hard to see that $R^{(n)}$ satisfies the recursion
\begin{equation} \label{eq:MaxRecSamplePath}
R^{(n)} = \left( \bigvee_{j=1}^{N_\emptyset} C_{(\emptyset, j)} R_j^{(n-1)} \right) \vee Q_\emptyset = \left( \bigvee_{j=1}^{N} C_{j} R_j^{(n-1)} \right) \vee Q ,
\end{equation}
where $\{R_j^{(n-1)} \}$ are independent copies of $R^{(n-1)}$ corresponding to the tree starting with individual $j$ in the first generation and ending on the $n$th generation. One can also verify that
$$V_n = \bigvee_{k=1}^{N_{\emptyset}} C_{(\emptyset,k)} \bigvee_{(k,\dots, i_n) \in A_n} 
Q_{(k,\dots, i_n)} \prod_{j=2}^n C_{(k,\dots,i_j)}  \stackrel{\mathcal{D}}{=} \bigvee_{k=1}^N C_k V_{(n-1),k},$$
where $\{V_{(n-1),k}\}$ is a sequence of iid random variables independent of $(N, C_1, C_2, \dots)$ and having the same distribution as $V_{n-1}$. 

We now define the random variable $R$ according to
\begin{equation}
\label{eq:maxR}
R \triangleq \lim_{n\to \infty} R^{(n)} = \bigvee_{k=0}^\infty V_k.
\end{equation}
Note that $R^{(n)}$ is monotone increasing sample-pathwise, so $R$ is well defined. Also, by monotonicity of $R^{(n)}$, \eqref{eq:MaxRecSamplePath} and monotone convergence, we obtain that $R$ solves
$$R = \left( \bigvee_{j=1}^{N_{\emptyset}} C_{(\emptyset, j)} R_j^{(\infty)} \right) \vee Q_{\emptyset} = \left( \bigvee_{j=1}^{N} C_{j} R_j^{(\infty)} \right) \vee Q,$$
where $\{R_j^{(\infty)} \}_{j \in \mathbb{N}}$ are iid copies of $R$, independent of $(Q, N, C_1, C_2, \dots)$. 
Clearly this implies that $R$, as defined by  \eqref{eq:maxR}, is a solution in distribution to \eqref{eq:Maximum}. However, this solution might be $\infty$. 
Now, we establish the finiteness of the moments of $R$, and in particular that $R < \infty$ a.s.,  in the following lemma; its proof uses standard contraction arguments but is included for completeness.

\begin{lem} \label{L.Moments_R_Max}
Assume that $\rho_\beta = E\left[ \sum_{i=1}^N C_i^\beta \right]<1$ and 
$E[Q^\beta] < \infty$ for some $\beta>0$. Then, $E[R^\gamma] < \infty$ for all $0  < \gamma \leq \beta$, and in particular, $R < \infty$ a.s. 
Moreover, if $\beta  \geq 1$, $R^{(n)} \stackrel{L_\beta}{\to} R$, where $L_\beta$ stands for convergence in $(E|\cdot|^\beta)^{1/\beta}$ norm. 
\end{lem}

\begin{proof}
By following the same steps leading to \eqref{eq:PiMoments}, we obtain that for any $k\ge 0$, 
\begin{equation}\label{eq:Vmoment}
E[V_k^\beta] = E\left[ \bigvee_{{\bf i} \in A_k} Q_{\bf i}^\beta \boldC_{\bf i}^\beta \right] \leq E\left[ \sum_{{\bf i} \in A_k} Q_{\bf i}^\beta \boldC_{\bf i}^\beta \right] = E[Q^\beta] \rho_\beta^k.\end{equation}
Hence,
$$E[R^\beta] = E\left[ \bigvee_{k=0}^\infty V_k^\beta \right] \leq E\left[ \sum_{k=0}^\infty V_k^\beta \right] \leq \frac{E[Q^\beta]}{1-\rho_\beta} < \infty,$$
implying that $E[R^\gamma] < \infty$ for all $0 < \gamma \leq \beta$. 

That $R^{(n)} \stackrel{L_\beta}{\to} R$ whenever $\beta\geq 1$ follows from noting that 
$E[|R^{(n)} - R|^\beta] \le E\left[ \left( \bigvee_{k = n+1}^\infty V_k \right)^\beta \right] \leq E\left[ \sum_{k = n+1}^\infty V_k^\beta  \right]$ 
and applying the preceding geometric bound for $E[V_k^\beta]$. 
\end{proof}

Just as with the linear recursion from Section \ref{S.LinearRec}, we can define the process $\{R_n^*\}$ as
\begin{equation*}
R_n^* \triangleq R^{(n-1)} \vee V_n(R_0^*), \qquad n \geq 1,
\end{equation*}
where
\begin{equation} \label{eq:MaxLastWeights}
V_n(R_0^*) = \bigvee_{{\bf i} \in A_n} R^*_{0,{\bf i}} \boldC_{{\bf i}},
\end{equation}
and $\{ R_{0,{\bf i}}^*\}_{{\bf i} \in U}$ are iid copies of an initial value $R_0^*$, independent of the entire weighted tree $\mathcal{T}_{Q,C}$. It follows from \eqref{eq:MaxRecSamplePath} and \eqref{eq:MaxLastWeights} that
\begin{equation*} 
R_{n+1}^* =  \bigvee_{j=1}^{N} C_j \left( R_{j}^{(n-1)}  \vee \bigvee_{{\bf i} \in A_{n,j}} R_{0,{\bf i}}^* \prod_{k=2}^n C_{(j,\dots,i_k)} \right) \vee Q = \bigvee_{j=1}^N C_j R_{n,j}^* \vee Q,
\end{equation*}
where $\{ R_{j}^{(n-1)} \}$ are independent copies of $R^{(n-1)}$ corresponding to the tree starting with individual $j$ in the first generation and ending on the $n$th generation, and $A_{n,j}$ is the set of all nodes in the $(n+1)$th generation that are descendants of individual $j$ in the first generation. Moreover, $\{R_{n,j}^*\}$ are iid copies of $R_n^*$, and thus, $R_n^*$ is equal in distribution to the process obtained by iterating \eqref{eq:Maximum} with an initial condition $R_0^*$. This process can be shown to converge in distribution to $R$ for any initial condition $R_0^*$ satisfying the following moment condition.

\begin{lem} \label{L.ConvergenceMax}
For any $R_0^* \geq 0$, if $E[Q^\beta], E[(R_0^*)^\beta] < \infty$ and $\rho_\beta < 1$ for some $\beta >0$, then
$$R_n^* \Rightarrow R,$$
with $E[R^\beta] < \infty$. Furthermore, under these assumptions, the distribution of $R$ is the unique solution with finite $\beta$-moment to recursion \eqref{eq:Maximum}. 
\end{lem}

\begin{proof}
The result will follow from Slutsky's Theorem (see Theorem 25.4, p. 332 in \cite{Billingsley_1995}) once we show that $V_n(R_0^*) \Rightarrow 0$. To this end, recall that $V_n(R_0^*)$ is the same as $V_n$ if we substitute the $Q_{{\bf i}}$ by the $R_{0,{\bf i}}^*$. Then, for every $\epsilon > 0$ we have that
\begin{align*}
P( V_n(R_0^*) > \epsilon) &\leq \epsilon^{-\beta} E[ V_n(R_0^*)^\beta] \\
&\leq \epsilon^{-\beta} \rho_\beta^n E[(R_0^*)^\beta] \qquad \text{(by \eqref{eq:Vmoment})} .
\end{align*}
Since by assumption the right-hand side converges to zero as $n \to \infty$, then $R_n^* \Rightarrow R$. Furthermore, $E[R^\beta] < \infty$ by Lemma \ref{L.Moments_R_Max}. Clearly, under the assumptions, the distribution of $R$ represents the unique solution to \eqref{eq:Maximum}, since any other possible solution with finite $\beta$-moment would have to converge to the same limit.  
\end{proof}

Now we are ready to state the main result of this section. 

\begin{thm} \label{T.MaximumRecursion}
Let $(Q, N, C_1, C_2, \dots)$ be a nonnegative random vector, with $N \in \mathbb{N} \cup \{\infty\}$, $P(Q > 0) > 0$
and $R$ be the solution to \eqref{eq:Maximum} given by  \eqref{eq:maxR}. 
Suppose that there exists $j \geq 1$ with $P(N\ge j,C_j>0)>0$ such that the measure $P(\log C_j\in du, C_j > 0, N\ge j)$ is nonarithmetic, and 
that for some $\alpha > 0$,  $E[Q^\alpha] < \infty$, $0 < E \left[ \sum_{i=1}^N C_i^\alpha \log C_i \right] < \infty$ and  $ E \left[ \sum_{i=1}^N C_i^\alpha \right] = 1$. In addition, assume
\begin{enumerate}
\item $E\left[ \left( \sum_{i=1}^N C_i \right)^\alpha \right] < \infty$, , if $\alpha > 1$; or,
\item $E\left[ \left( \sum_{i=1}^N C_i^{\alpha/(1+\epsilon)}\right)^{1+\epsilon} \right] < \infty$ for some $0 < \epsilon< 1$, if $0 < \alpha \leq 1$.
\end{enumerate}
Then,
$$P(R > t) \sim H t^{-\alpha}, \qquad t \to \infty,$$
where $0 \leq H < \infty$ is given by
\begin{align*}
H &= \frac{1}{E\left[ \sum_{i=1}^N C_i^\alpha \log C_i  \right] } \int_{0}^\infty v^{\alpha-1} \left( P(R > v) - E\left[ \sum_{i=1}^{N} \Indicator(C_{i} R > v ) \right]    \right) dv \\
&= \frac{E\left[ \left( \bigvee_{i=1}^N C_i R_i \right)^\alpha \vee Q^\alpha - \sum_{i=1}^N (C_i R_i )^\alpha \right]}{\alpha E\left[ \sum_{i=1}^N C_i^\alpha \log C_i  \right] }.
\end{align*}
\end{thm}

\begin{proof}
By Lemma \ref{L.Moments_R_Max} we know that $E[R^\beta] < \infty$ for any $0 < \beta < \alpha$. The same arguments used in the proof of Theorem \ref{T.LinearRecursion} give that $E\left[ \sum_{i=1}^N C_i^\gamma \right] < \infty$ for some $0 \leq \gamma <  \alpha$. The statement of the theorem with the first expression for $H$ will follow from Theorem \ref{T.NewGoldie} once we prove that condition \eqref{eq:Goldie_condition} holds.  Define
$$R^* = \left( \bigvee_{i=1}^N C_i R_i \right) \vee Q.$$
Then, 
\begin{align*}
\left| P(R>t) - E\left[ \sum_{i=1}^N \Indicator(C_iR_i > t ) \right] \right|  &\leq \left| P(R > t) - P\left( \max_{1\leq i \leq N} C_i R_i > t \right) \right|   \\
&\hspace{5mm} + \left| P\left( \max_{1\leq i \leq N} C_i R_i > t  \right)   - E\left[ \sum_{i=1}^N \Indicator(C_iR_i > t ) \right]  \right|.
\end{align*}
Since $R \stackrel{\mathcal{D}}{=} R^* \geq \max_{1\leq i\leq N} C_i R_i$, the first absolute value disappears. The integral corresponding to the second term is finite by Lemma \ref{L.Max_Approx}, just as in the proof of Theorem~\ref{T.LinearRecursion}. To see that the integral corresponding to the first term, 
$$\int_0^\infty  \left( P(R > t) - P\left( \max_{1\leq i \leq N} C_i R_i > t \right) \right) t^{\alpha-1} \, dt, $$ 
is finite we proceed as in the proof of Theorem \ref{T.LinearRecursion}. First we use Fubini's Theorem to obtain that
\begin{align*}
&\int_0^\infty  \left( P(R > t) - P\left( \max_{1\leq i \leq N} C_i R_i > t \right) \right) t^{\alpha-1} \, dt  \\
&= \frac{1}{\alpha} E \left[ (R^*)^\alpha - \left( \max_{1\leq i \leq N} C_i R_i \right)^\alpha   \right]  \\
&= \frac{1}{\alpha} E \left[ \left( \bigvee_{i=1}^N C_i R_i \right)^\alpha \vee Q^\alpha - \left( \bigvee_{i=1}^N C_i R_i \right)^\alpha   \right] \\
&\leq \frac{E[Q^\alpha]}{\alpha}.
\end{align*}

Now, applying Theorem \ref{T.NewGoldie} gives 
$$P(R > t) \sim H t^{-\alpha},$$
where $H = \left( E\left[ \sum_{j=1}^N C_j^\alpha \log C_j  \right] \right)^{-1} \int_{0}^\infty v^{\alpha-1} \left( P(R > v) - E\left[ \sum_{j=1}^{N} \Indicator(C_{j} R > v ) \right]    \right) dv$. 

The same steps used in the proof of Theorem \ref{T.LinearRecursion} give the second expression for $H$.
\end{proof}

%%% Other recursions

\section{Other recursions and concluding remarks}

As an illustration of the generality of the methodology presented in this paper, we mention in this section other recursions that fall within its scope. One example that is closely related to the recursions from Sections~\ref{S.LinearRec} and \ref{S.MaxRec} is the following
\begin{equation} \label{eq:Max_Sum_Rec}
R \stackrel{\mathcal{D}}{=} \left( \bigvee_{i=1}^N C_i R_i \right) + Q,
\end{equation}
where $(Q, N, C_1, C_2, \dots)$ is a nonnegative vector with $N \in \mathbb{N} \cup \{\infty\}$, $P(Q > 0) > 0$, and $\{R_i \}_{i \in \mathbb{N}}$ is a sequence of iid random variables independent of $(Q, N, C_1, C_2, \dots)$ having the same distribution as $R$.  Recursion \eqref{eq:Max_Sum_Rec} was termed ``discounted tree sums" in \cite{Aldo_Band_05}; for additional details on the existence and uniqueness of its solution see Section 4.4 in \cite{Aldo_Band_05}.

Similarly as in \cite{Goldie_91}, it appears that one could study other non-linear recursions on trees using implicit renewal theory. For example, one could analyze the solution to the equation 
$$R \stackrel{\mathcal{D}}{=} \sum_{i=1}^N \left( C_i R_i + B_i \sqrt{R_i} \right) + Q,$$
where $(Q, N, C_1, C_2, \dots)$ is a nonnegative vector with $N \in \mathbb{N} \cup \{\infty\}$, $P(Q > 0) > 0$, and $\{R, R_i \}_{i \geq 1}$ is a sequence of iid random variables independent of $(Q, N, C_1, C_2, \dots)$. Here, the primary difficulty would be in establishing the existence and uniqueness of the solution as well as the finiteness of moments.

%%% Proofs
\section{Proofs} \label{S.Proofs}

\subsection{Implicit renewal theorem on trees} \label{SS.ImplicitProofs}

We give in this section the proof of Lemma \ref{L.RenewalMeasure}. 

\begin{proof}[Proof of Lemma \ref{L.RenewalMeasure}]
Observe that the measure $E\left[ \sum_{i=1}^N \Indicator(\log C_i \in du, C_i>0)  \right]$ is nonarithmetic (nonlattice) by our assumption
since, if we assume to the contrary that it is lattice on a lattice set $L$, then on the complement $L^c$ of this set we have 
$$0=E\left[ \sum_{i=1}^N \Indicator(\log C_i \in L^c, C_i>0 )  \right]\ge P(\log C_j \in L^c, C_j>0,N \ge j)>0,$$ 
which is a contradiction.  Hence, $\eta$ is nonarithmetic as well, and it places no mass at $-\infty$ due to the exponential term $e^{\alpha u}$.
To see that it is a probability measure note that 
\begin{align*}
\int_{-\infty}^{\infty} \eta(du) &=  \int_{-\infty}^\infty e^{\alpha u}E\left[ \sum_{j=1}^N \Indicator(\log C_j \in du )  \right] \\
&= E\left[ \sum_{j=1}^N \int_{-\infty}^\infty e^{\alpha u} \Indicator(\log C_j \in du )  \right] \qquad \text{(by Fubini's Theorem)} \\
&=  E\left[ \sum_{j=1}^N C_j^\alpha \right] = 1.
\end{align*}  
Similarly, its mean is given by
$$\int_{-\infty}^\infty u \eta(du) = E\left[ \sum_{j=1}^N C_j^\alpha \log C_j  \right] .$$

To show that $\mu_n = \eta^{*n}$ we proceed by induction. Let $\mathcal{F}_n$ denote the $\sigma$-algebra generated by $\left\{ (N_{\bf i}, C_{({\bf i}, 1)}, C_{({\bf i}, 2)}, \dots) : {\bf i} \in A_j, 0 \leq j \leq n-1 \right\}$, $\mathcal{F}_0 = \sigma(\emptyset, \Omega)$, and for each ${\bf i} \in A_n$ set $V_{{\bf i}} = \log \boldC_{{\bf i}}$.  Hence, using this notation we derive
\begin{align*}
\mu_{n+1}((-\infty,t]) &= \int_{-\infty}^t e^{\alpha u} E\left[  \sum_{{\bf i} \in A_{n}} \sum_{j=1}^{N_{{\bf i}}} \Indicator(V_{{\bf i}} + \log C_{({\bf i},j)} \in du )   \right] \\
&= \int_{-\infty}^t e^{\alpha u} E\left[  \sum_{{\bf i} \in A_{n}} E\left[ \left. \sum_{j=1}^{N_{{\bf i}}} \Indicator(V_{{\bf i}} + \log C_{({\bf i},j)} \in du ) \right| \mathcal{F}_n \right]  \right] \\
&=  E\left[  \sum_{{\bf i} \in A_{n}} e^{\alpha V_{\bf i}} \int_{-\infty}^t e^{\alpha (u- V_{\bf i})} E\left[ \left. \sum_{j=1}^{N_{{\bf i}}} \Indicator( \log C_{({\bf i},j)} \in du - V_{\bf i} ) \right| \mathcal{F}_n \right]  \right] \\
&=E\left[  \sum_{{\bf i} \in A_{n}} e^{\alpha V_{\bf i}} \eta((-\infty, t - V_{{\bf i}}])  \right] \\
&= \int_{-\infty}^\infty  \eta((-\infty,t-v])  \mu_n(dv),
\end{align*}
where in the fourth equality we used the independence of $(N_{{\bf i}}, C_{({\bf i}, 1)}, C_{({\bf i}, 2)}, \dots)$ from $\mathcal{F}_n$. Therefore $\mu_{n+1}(dt) = (\eta*\mu_n)(dt)$ and the induction hypothesis gives the result.
\end{proof}

%\begin{proof}[Proof of Lemma \ref{L.Derivative}]
%This lemma is a special case of the Monotone Density Theorem, see Theorem 1.7.5 (also Exercise~1.11.14) in \cite{BiGoTe1987}. However, for completeness, we give a direct proof here, similar to the one of Lemma~9.3 in \cite{Goldie_91}. By assumption, for any $b > 1$, $\epsilon \in (0,1)$, and $t$ sufficiently large,
%\begin{align*}
%P(R > t) t^{\alpha +\beta} \cdot \frac{b^{\alpha+\beta}-1}{\alpha+\beta} &\geq \int_{t}^{b t} v^{\alpha+\beta-1} P(R > v) \, dv \geq \frac{(H-\epsilon)}{\beta} (b t)^\beta - \frac{(H+\epsilon)}{\beta} t^\beta  \\
%&\geq \frac{t^\beta}{\beta} \left( H (b^\beta-1) -\epsilon (1 + b^\beta)  \right).
%\end{align*}
%Since $\epsilon$ was arbitrary, we can take the limit as $\epsilon \to 0$ and obtain
%$$\liminf_{t \to \infty} P(R > t) t^{\alpha} \geq \frac{H (\alpha+\beta) (b^\beta-1)}{\beta(b^{\alpha+\beta} - 1)} \to H, \qquad b \downarrow 1.$$
%Similarly, one can prove that $\limsup_{t \to \infty} P(R > t) t^\alpha \leq H$ starting from $\int_{bt}^t v^{\alpha+\beta -1} P(R > v) \, dv$ with $0 < b < 1$. 
%\end{proof}

%%% Proofs for moments
\subsection{Moments of $W_n$} \label{SS.MomentsProofs}

In this section we prove Lemmas \ref{L.Alpha_Moments}, \ref{L.MomentSmaller_1} and \ref{L.GeneralMoment}.  We also include a result that provides bounds for $E[W_n^p]$ for integer $p$, which will be used in the proof of Lemma \ref{L.GeneralMoment}. 

\begin{proof}[Proof of Lemma \ref{L.Alpha_Moments}]
Let $p = \lceil \beta \rceil \in \{2,3,\dots\}$ and $\gamma = \beta/p \in (\beta/(\beta+1), 1]$. Suppose first that $k \in \mathbb{N}$ and define  $A_p(k) = \{ (j_1, \dots, j_k) \in \mathbb{N}^k: j_1 + \dots + j_k = p, 0 \leq j_i < p\}$. Then, for any sequence of nonnegative numbers $\{ y_i \}_{i \geq 1}$ we have
\begin{align}
\left( \sum_{i=1}^k y_i \right)^\beta &= \left( \sum_{i=1}^k y_i \right)^{p \gamma} \notag \\
&= \left( \sum_{i=1}^k y_i^p + \sum_{(j_1,\dots,j_k) \in A_p(k)} \binom{p}{j_1,\dots,j_k} y_1^{j_1} \cdots y_k^{j_k} \right)^\gamma \notag  \\
&\leq \sum_{i=1}^k y_i^{p\gamma} + \left( \sum_{(j_1,\dots,j_k) \in A_p(k)} \binom{p}{j_1,\dots,j_k} y_1^{j_1} \cdots y_k^{j_k} \right)^\gamma, \label{eq:scalarBound}
\end{align}
where for the last step we used the well known inequality $\left( \sum_{i=1}^k x_i \right)^\gamma \leq \sum_{i=1}^k x_i^\gamma$ for $0 < \gamma \leq 1$ and $x_i \geq 0$. We now use the conditional Jensen's inequality to obtain 
\begin{align*}
&E\left[ \left( \sum_{i=1}^k C_i Y_i \right)^\beta - \sum_{i=1}^k (C_iY_i)^{\beta} \right] \\
&\leq E\left[  \left( \sum_{(j_1,\dots,j_k) \in A_p(k)} \binom{p}{j_1,\dots,j_k} (C_1Y_1)^{j_1} \cdots (C_k Y_k)^{j_k} \right)^\gamma \right] \qquad \text{(by \eqref{eq:scalarBound})} \\
&\leq    E\left[ \left(  E\left[ \left. \sum_{(j_1,\dots,j_k) \in A_p(k)} \binom{p}{j_1,\dots,j_k} (C_1Y_1)^{j_1} \cdots (C_k Y_k)^{j_k} \right| C_1,\dots, C_k \right] \right)^\gamma \right]  \\
&= E \left[ \left(  \sum_{(j_1,\dots,j_k) \in A_p(k)} \binom{p}{j_1,\dots,j_k} C_1^{j_1} \cdots C_k^{j_k} E\left[ \left. Y_1^{j_1} \cdots Y_k^{j_k} \right| C_1,\dots, C_k \right] \right)^\gamma  \right].
\end{align*}
Since $\{Y_i\}$ is a sequence of iid random variables having the same distribution as $Y$, independent of the $C_i$'s we have that
$$E\left[ \left. Y_1^{j_1} \cdots Y_k^{j_k} \right| C_1,\dots, C_k \right] = E\left[ Y_1^{j_1} \cdots Y_k^{j_k} \right]  = || Y ||_{j_1}^{j_1}  \cdots || Y ||_{j_k}^{j_k},$$
where $|| Y ||_\kappa = \left( E[|Y|^\kappa] \right)^{1/\kappa}$ for $\kappa \geq 1$ and $|| Y ||_0 \triangleq 1$. Since $|| Y ||_\kappa$ is increasing for $\kappa \geq 1$ it follows that $|| Y ||_{j_i}^{j_i} \leq || Y ||_{p-1}^{j_i}$. Hence
$$|| Y ||_{j_1}^{j_1} \cdots || Y ||_{j_k}^{j_k} \leq || Y ||_{p-1}^p,$$
which in turn implies that
\begin{align*}
E\left[ \left( \sum_{i=1}^k C_iY_i \right)^\beta - \sum_{i=1}^k (C_i Y_i)^{\beta} \right] &\leq E \left[ \left(  \sum_{(j_1,\dots,j_k) \in A_p(k)} \binom{p}{j_1,\dots,j_k} C_1^{j_1} \cdots C_k^{j_k} || Y ||_{p-1}^p \right)^\gamma  \right] \\
&= || Y ||_{p-1}^\beta E\left[ \left( \left(\sum_{i=1}^k C_i \right)^p - \sum_{i=1}^k C_i^p \right)^\gamma  \right]  \\
&\leq || Y ||_{p-1}^\beta E\left[ \left(\sum_{i=1}^k C_i \right)^\beta  \right].
\end{align*}
This completes the proof for $k$ finite. 

When $k = \infty$, first note that from the well known inequality $\left(  x_1 + x_2 \right)^\beta \geq x_1^\beta + x_2^\beta$ for $x_1, x_2 \geq 0$ and $\beta > 1$ we obtain the monotonicity in $k$ of the following difference 
$$\left( \sum_{i=1}^{k+1} C_i Y_i \right)^\beta - \sum_{i=1}^{k+1} (C_i Y_i)^\beta \geq  \left( \sum_{i=1}^{k} C_i Y_i \right)^\beta - \sum_{i=1}^k (C_i Y_i)^\beta \geq 0.$$
Hence,
\begin{align}
&E\left[ \left( \sum_{i=1}^\infty C_iY_i \right)^\beta - \sum_{i=1}^\infty (C_i Y_i)^{\beta} \right] \notag \\
&= \lim_{k\to \infty} E\left[  \left( \left( \sum_{i=1}^k C_iY_i \right)^\beta - \sum_{i=1}^k (C_i Y_i)^{\beta} \right) \right] \label{eq:Exchange1} \\
&\leq \lim_{k\to \infty} E\left[  \left( \sum_{(j_1,\dots,j_k) \in A_p(k)} \binom{p}{j_1,\dots,j_k} (C_1Y_1)^{j_1} \cdots (C_k Y_k)^{j_k} \right)^\gamma  \right]  \notag \\
&\leq \lim_{k \to \infty} || Y ||_{p-1}^\beta  E\left[ \left(\sum_{i=1}^k C_i \right)^\beta  \right]  \notag \\
&= || Y ||_{p-1}^\beta  E\left[ \left(\sum_{i=1}^\infty C_i \right)^\beta  \right] , \label{eq:Exchange2}
\end{align}
where \eqref{eq:Exchange1} and \eqref{eq:Exchange2} are justified by monotone convergence. 
\end{proof}

\begin{proof}[Proof of Lemma \ref{L.MomentSmaller_1}]
We use the well known inequality $\left( \sum_{i=1}^k y_i \right)^\beta \leq \sum_{i=1}^k y_i^\beta$ for $0 < \beta \leq 1$, $y_i \geq 0$, $k \leq \infty$, to obtain
\begin{align}
E[W_n^\beta] &= E\left[ \left( \sum_{i=1}^N C_i W_{(n-1),i} \right)^\beta \right] \notag \\
&\leq E\left[ \sum_{i=1}^N C_i^\beta  W_{(n-1),i}^\beta \right] \notag \\
&= E[W_{n-1}^\beta] \rho_\beta \qquad \text{(by conditioning on $N, C_i$ and Fubini's theorem)} \notag \\
&\leq \rho_\beta^{n} E[W_0^\beta] \qquad \text{(iterating $n$ times)} \notag \\
&= \rho_\beta^{n} E[Q^\beta] . 
\end{align}
\end{proof}

Before proving the moment inequality for  general $\beta > 1$, we will show first the corresponding result for integer moments.

\begin{lem} \label{L.IntegerMoment}
Let $p \in \{2, 3, \dots\}$ and recall that $\rho_p = E\left[ \sum_{i=1}^N C_i^p \right]$, $\rho \equiv \rho_1$. Suppose $E[Q^p]< \infty$, $E\left[ \left( \sum_{i=1}^N C_i \right)^p \right] < \infty$, and $\rho \vee \rho_p < 1$. Then, there exists a constant $K_p > 0$ such that
$$E[ W_n^p ] \leq K_p \left( \rho \vee \rho_p \right)^n $$
for all $n \geq 0$.  
\end{lem}

\begin{proof}
We will give an induction proof in $p$. For $p = 2$ we have
\begin{align*}
E[W_n^2] &= E\left[ \left( \sum_{i=1}^N C_i W_{(n-1),i} \right)^2  \right] \\
&= E\left[ \sum_{i=1}^N C_i^2 W_{(n-1),i}^2 + \sum_{i \neq j} C_i W_{(n-1),i} C_j W_{(n-1),j}  \right] \\
&= E[W_{n-1}^2] E\left[ \sum_{i=1}^N C_i^2 \right] + \left( E[W_{n-1}] \right)^2 E\left[ \sum_{i \neq j} C_i  C_j  \right] \\
&\hspace{4.5cm} \text{(by conditioning on $N, C_i$ and Fubini's theorem)} \\
&\leq \rho_2 E[W_{n-1}^2] + E\left[ \left( \sum_{i=1}^N C_i \right)^2 \right] \left( E[W_{n-1}] \right)^2 .
\end{align*}

Using the preceding recursion and noting that,
$$E[W_{n-1}] = \rho^{n-1} E[Q],$$
we obtain
\begin{equation} \label{eq:2_recur}
E[W_n^2] \leq \rho_2 E[W_{n-1}^2] + K \rho^{2(n-1)},
\end{equation}
where $K =E\left[ \left( \sum_{i=1}^N C_i \right)^2 \right] \left( E[Q] \right)^2$. Now, iterating \eqref{eq:2_recur} gives
\begin{align*}
E[W_n^2] &\leq \rho_2 \left( \rho_2 E[W_{n-2}^2] +  K \rho^{2(n-2)} \right) +  K \rho^{2(n-1)} \\
&\leq \rho_2^{n-1} \left( \rho_2 E[W_{0}^2] + K  \right) + K \sum_{i=0}^{n-2} \rho_2^i \, \rho^{2(n-1-i)} \\
&= \rho_2^n E[Q^2] + K  \sum_{i=0}^{n-1} \rho_2^i \, \rho^{2(n-1-i)}  \\
&\leq (\rho_2 \vee \rho)^n E[Q^2] + K (\rho_2 \vee \rho)^n \sum_{i=0}^{n-1} (\rho_2 \vee \rho)^{n-2 - i  }  \\
&\leq \left( E[Q^2] + \frac{K}{\rho_2 \vee \rho} \sum_{j=0}^{\infty} (\rho_2 \vee \rho)^{j}  \right) (\rho_2 \vee \rho)^n \\
&= K_2 (\rho_2 \vee \rho)^n ,
\end{align*}
which completes the case $p = 2$. 

Suppose now that there exists a constant $K_{p-1} > 0$ such that
\begin{equation} \label{eq:Induction_p}
E[W_n^{p-1}] \leq K_{p-1} \left( \rho_{p-1} \vee \rho \right)^n
\end{equation}
for all $n \geq 0$. Then, by Lemmas \ref{L.Alpha_Moments} and \ref{L.MomentSmaller_1}, we have
\begin{align*}
E[W_n^p] &= E\left[ \left( \sum_{i=1}^N C_i W_{(n-1),i} \right)^p - \sum_{i=1}^N C_i^p W_{(n-1),i}^p \right] + E\left[  \sum_{i=1}^N C_i^p W_{(n-1),i}^p \right] \\
&\leq \left( E\left[ W_{n-1}^{p-1} \right] \right)^{p/(p-1)} E\left[ \left( \sum_{i=1}^N C_i \right)^p \right] + E\left[  \sum_{i=1}^N C_i^p W_{(n-1),i}^p \right] \\
&= \left( E\left[ W_{n-1}^{p-1} \right] \right)^{p/(p-1)} E\left[ \left( \sum_{i=1}^N C_i \right)^p \right] + \rho_p E\left[ W_{n-1}^{p} \right] \\
&\leq  E\left[ \left( \sum_{i=1}^N C_i \right)^p \right]  (K_{p-1})^{p/(p-1)} (\rho_{p-1} \vee \rho)^{(n-1)p/(p-1)} + \rho_p E[W_{n-1}^p],
\end{align*}
where in the second equality we conditioned on $N, C_i$ and used Fubini's theorem, and the last inequality corresponds to the induction hypothesis. We then obtain the recursion
\begin{equation} \label{eq:p_recur}
E[W_n^p] \leq \rho_p E[W_{n-1}^p] + K (\rho_{p-1} \vee \rho)^{\frac{(n-1)p}{p-1}},
\end{equation}
where $K = E\left[ \left( \sum_{i=1}^N C_i \right)^p \right]  (K_{p-1})^{p/(p-1)}$. Iterating \eqref{eq:p_recur} as for the case $p=2$ gives
\begin{align}
E[W_n^p] &\leq \rho_p^n E[Q^p] + K  \sum_{i=0}^{n-1} \rho_p^i \, (\rho_{p-1} \vee \rho)^{(n-1-i)p/(p-1)} \notag \\
&\leq (\rho_p \vee \rho)^n E[Q^p] + K \sum_{i=0}^{n-1} (\rho_p \vee \rho)^{((n-1)p -i)/(p-1) } \label{eq:MGFconvexity}  \\
\phantom{E[W_n^p]}&= (\rho_p \vee \rho)^n E[Q^p] + K (\rho_p \vee \rho)^{n-1}   \sum_{i=0}^{n-1} (\rho_p \vee \rho)^{(n- i - 1)/(p-1)}  \notag  \\
&\leq  \left( E[Q^p] + K (\rho_p \vee \rho)^{-1}   \sum_{j=0}^{\infty} (\rho_p \vee \rho)^{\frac{j}{p-1}} \right) (\rho_p \vee \rho)^n \notag \\
&= K_p (\rho_p \vee \rho)^n, \notag
\end{align}
where in \eqref{eq:MGFconvexity} we used the convexity of $\varphi(\beta) = \rho_\beta$, i.e. $\rho_{p-1} = \varphi(p-1) \leq \varphi(1) \vee \varphi(p) = \rho \vee \rho_p$. 
\end{proof}

The proof for the general $\beta$-moment, $\beta > 1$, is given below.

\begin{proof}[Proof of Lemma \ref{L.GeneralMoment}]
Set $p = \lceil \beta \rceil \geq \beta > 1$. Then, by Lemmas \ref{L.Alpha_Moments} and \ref{L.MomentSmaller_1}, 
\begin{align*}
E[W_n^\beta] &= E\left[ \left( \sum_{i=1}^N C_i W_{(n-1),i} \right)^\beta - \sum_{i=1}^N C_i^\beta W_{(n-1),i}^\beta \right] + E\left[  \sum_{i=1}^N C_i^\beta W_{(n-1),i}^\beta \right] \\
&\leq \left( E\left[ W_{n-1}^{p-1} \right] \right)^{\beta/(p-1)} E\left[ \left( \sum_{i=1}^N C_i \right)^\beta \right] + E\left[  \sum_{i=1}^N C_i^\beta W_{(n-1),i}^\beta \right] \\
&= \left( E\left[ W_{n-1}^{p-1} \right] \right)^{\beta/(p-1)} E\left[ \left( \sum_{i=1}^N C_i \right)^\beta \right] + \rho_\beta E\left[ W_{n-1}^{\beta} \right]   .
\end{align*}
By Lemma \ref{L.IntegerMoment}, 
\begin{align*}
E[W_n^\beta] &\leq  \rho_\beta E[W_{n-1}^\beta] + E\left[ \left( \sum_{i=1}^N C_i \right)^\beta \right]  (K_{p-1})^{\beta/(p-1)} (\rho_{p-1} \vee \rho)^{(n-1)\beta/(p-1)} \\
&= \rho_\beta E[ W_{n-1}^\beta] + K (\rho_{p-1} \vee \rho)^{(n-1)\gamma},
\end{align*}
where $\gamma = \beta/(p-1) > 1$. Finally, iterating the preceding bound $n-1$ times gives
\begin{align*}
E[W_n^\beta] &\leq \rho_\beta^n E[W_0^\beta] + K \sum_{i=0}^{n-1} \rho_\beta^i (\rho \vee \rho_{p-1})^{\gamma(n-1-i)} \\
&\leq E[W_0^\beta]  (\rho \vee \rho_\beta)^n + K  \sum_{i=0}^{n-1} (\rho \vee \rho_\beta)^{\gamma(n-1-i) + i} \qquad \text{(by convexity of $\varphi(\beta) = \rho_\beta$)} \\
&= E[Q^\beta] (\rho \vee \rho_\beta)^n + K (\rho \vee \rho_\beta)^{n-1} \sum_{i=0}^{n-1} (\rho \vee \rho_\beta)^{(\gamma-1) i} \\
&\leq K_\beta (\rho \vee \rho_\beta)^n .
\end{align*}
This completes the proof. 
\end{proof}

%%% Proofs for linear recursion
\subsection{Linear nonhomogeneous recursion} \label{SS.LinearProofs}

In this section we give the proofs of the technical Lemmas \ref{L.Max_Approx} and \ref{L.ExtraQ} for the linear recursion. 

\begin{proof}[Proof of Lemma \ref{L.Max_Approx}]
Note that the integral is positive since 
\begin{align*}
P\left( \max_{1\leq i \leq N} C_i R_i > t \right) = E\left[   \Indicator\left(  \max_{1\leq i \leq N} C_i R_i > t \right)  \right] &\leq E\left[   \sum_{i=1}^N \Indicator\left( C_i R_i > t  \right)  \right] .
\end{align*}

To see that the integral is equal to the expectation involving the $\alpha$-moments note that
\begin{align*}
&\int_{0}^\infty \left( E\left[ \sum_{i=1}^N \Indicator(C_i R_i > t) \right] - P\left( \max_{1\leq i \leq N} C_i R_i > t \right) \right)  t^{\alpha -1} \, dt \\
&= \int_0^\infty E\left[ \sum_{i=1}^N \Indicator(C_i R_i > t) - \Indicator\left(\max_{1\leq i \leq N} C_i R_i > t \right)   \right]  t^{\alpha -1} \, dt \\
&= E\left[  \int_0^\infty \left(   \sum_{i=1}^N  \Indicator(C_i R_i > t)  - \Indicator\left( \max_{1\leq i \leq N} C_i R_i > t \right) \right)  t^{\alpha -1} \, dt   \right] \qquad \text{(by Fubini's theorem)} \\
&= E\left[   \sum_{i=1}^N  \frac{1}{\alpha} (C_i R_i)^{\alpha}  - \frac{1}{\alpha} \left(\max_{1\leq i \leq N} C_i R_i \right)^{\alpha}  \right] ,
\end{align*}
where the last equality is justified by the assumption that $\sum_{i=1}^N (C_i R_i)^\alpha < \infty$ a.s.

It now remains to show that the integral (expectation) is finite. To do this let ${\bf X} = (N, C_1, C_2, \dots)$. Similar arguments to those used above give
\begin{align*}
&\int_{0}^\infty \left( E\left[ \sum_{i=1}^N \Indicator(C_i R_i > t ) \right] - P\left( \max_{1\leq i \leq N} C_i R_i > t \right) \right)  t^{\alpha -1} \, dt \\
&= \int_0^\infty E\left[ E\left[ \left. \sum_{i=1}^N \Indicator(C_i R_i > t) - \Indicator\left(\max_{1\leq i \leq N} C_i R_i > t \right) \right| {\bf X} \right]  \right]   t^{\alpha -1} \, dt \\
&= E\left[  \int_0^\infty E\left[ \left. \sum_{i=1}^N \Indicator(C_i R_i > t) - \Indicator\left(\max_{1\leq i \leq N} C_i R_i > t \right) \right| {\bf X} \right]     t^{\alpha -1} \, dt    \right], 
\end{align*}
where in the last step we used Fubini's theorem. Furthermore,
\begin{align*}
&E\left[ \left. \sum_{i=1}^N \Indicator(C_i R_i > t) - \Indicator\left(\max_{1\leq i \leq N} C_i R_i > t \right) \right| {\bf X} \right]   \\
&= E\left[ \left. \Indicator\left( \max_{1\leq i \leq N} C_i R_i \leq t \right) \right|  {\bf X} \right] - 1 + \sum_{i=1}^N  E\left[ \Indicator(C_i R_i > t) | {\bf X}\right]. 
\end{align*}

Note that given ${\bf X}$, the random variables $C_i R_i$ are independent (since the $R$'s are), so if we let $\overline{F}(t) = P(R > t)$, then
$$E\left[ \left. \Indicator\left( \max_{1\leq i \leq N} C_i R_i \leq t \right) \right|  {\bf X} \right] = \prod_{i=1}^N E\left[ \left.  \Indicator\left( C_i R_i \leq t \right) \right|  {\bf X} \right] = \prod_{i=1}^N \left(1 - \overline{F}(t/C_i) \right).$$
We now use the inequality $1 - x \leq e^{-x}$ for $x \geq 0$ to obtain
$$\prod_{i=1}^N \left(1 - \overline{F}(t/C_i) \right) \leq e^{-\sum_{i=1}^N \overline{F}(t/C_i)}.$$
Next, let $\delta = \alpha\epsilon/(1+\epsilon)$ and set $\beta = \alpha-\delta$. By Markov's inequality,
$$\sum_{i=1}^N \overline{F}(t/C_i) \leq \sum_{i=1}^N E[(C_i R)^\beta | C_i] t^{-\beta} = t^{-\beta} E[R^\beta] \sum_{i=1}^N C_i^\beta.$$

Now, define the function $g(x) = e^{-x} - 1+ x$ and note that $g(x)$ is increasing for $x \geq 0$. Therefore, 
$$g\left( \sum_{i=1}^N \overline{F}(t/C_i) \right) \leq g\left(t^{-\beta} E[R^\beta] \sum_{i=1}^N C_i^\beta \right).$$
This observation combined with the previous derivations gives
\begin{align*}
&\int_0^\infty  E\left[ \left. \sum_{i=1}^N \Indicator(C_i R_i > t) - \Indicator\left(\max_{1\leq i \leq N} C_i R_i > t \right) \right| {\bf X} \right]    t^{\alpha -1} \, dt  \\
&\leq \int_0^\infty \left( e^{-r S_\beta t^{-\beta}} - 1 + r S_\beta t^{-\beta}  \right) t^{\alpha-1} dt,
\end{align*}
where $S_\beta = \sum_{i=1}^N C_i^\beta$ and $r = E[R^\beta] < \infty$. Hence, using the change of variables $u =r S_\beta t^{-\beta}$ gives
\begin{align*}
\int_{0}^\infty \left( e^{-rS_\beta t^{-\beta}}  - 1 + r S_\beta t^{-\beta} \right)  t^{\alpha -1} \, dt &= \beta^{-1} (r S_\beta)^{\alpha/\beta} \int_0^\infty \left( e^{-u} - 1 + u \right) u^{-\alpha/\beta  -1} \, du.
\end{align*}
Our choice of $\beta = \alpha-\delta$ guarantees that $1 < \alpha/\beta = 1+\epsilon < 2$. To see that the (non-random) integral is finite note that
$e^{-x} -1 + x \leq x^2/2$ and $e^{-x} - 1 \leq 0$ for any $x \geq 0$, implying
\begin{align*}
\int_0^\infty \left( e^{-u} - 1 + u \right) u^{-\alpha/\beta  -1} \, du &\leq \frac{1}{2} \int_0^1  u^{1-\alpha/\beta} \, du + \int_1^\infty  u^{-\alpha/\beta } \, du \\
&= \frac{1}{2(2-\alpha/\beta)} + \frac{1}{\alpha/\beta-1} \triangleq K \beta < \infty.
\end{align*}
Now, it follows that
\begin{align*}
&\int_{0}^\infty \left( E\left[ \sum_{i=1}^N \Indicator(C_i R_i > t) \right] - P\left( \max_{1\leq i \leq N} C_i R_i > t \right) \right)  t^{\alpha -1} \, dt  \\
&\leq E\left[  K (r S_\beta)^{\alpha/\beta}  \right] = K  r^{\alpha/\beta} E\left[ \left( \sum_{i=1}^N C_i^\beta  \right)^{\alpha/\beta}  \right] .
\end{align*}
The last expectation is finite by assumption ($\alpha/\beta = 1 + \epsilon$), which completes the proof. 
\end{proof}

\begin{proof}[Proof of Lemma \ref{L.ExtraQ}]
Let $S = \sum_{i=1}^N  C_i R_i < \infty$ a.s., $p = \lceil \alpha \rceil$ and note that $1 \leq p-1 < \alpha$. Then, since $(S+Q)^\alpha - S^\alpha \geq 0$ and $S^\alpha - \sum_{i=1}^N (C_i R_i)^\alpha \geq 0$, we can break the expectation as follows
\begin{align*}
E \left[ (S+Q)^\alpha - \sum_{i=1}^N  \left(C_i R_i \right)^\alpha   \right] &= E\left[ (S+Q)^\alpha - S^\alpha \right] + E\left[ \left( \sum_{i=1}^N C_i R_i \right)^\alpha - \sum_{i=1}^N (C_iR_i)^\alpha \right] \\
&\leq E\left[ (S+Q)^\alpha - S^\alpha \right]  + \left( E[ R^{p-1} ] \right)^{\alpha/(p-1)}  E\left[ \left(\sum_{i=1}^N C_i \right)^\alpha  \right] ,
\end{align*}
where the inequality is justified by Lemma \ref{L.Alpha_Moments}.  The second expectation is finite since by assumption $E[R^\beta] < \infty$ for any $0 <\beta < \alpha$. For the first expectation we use the inequality 
$$(x+t)^\kappa \leq \begin{cases}
x^\kappa + t^\kappa, & 0 < \kappa \leq 1, \\
x^\kappa + \kappa (x+t)^{\kappa-1} t, & \kappa > 1,
\end{cases}$$
for any $x,t \geq 0$. We apply the second inequality $p-1$ times and then the first one to obtain 
\begin{align*}
(x+t)^\alpha \leq x^\alpha + \alpha (x+t)^{\alpha-1} t  \leq \dots &\leq x^\alpha + \sum_{i=1}^{p-2} \alpha^i x^{\alpha-i} t^i + \alpha^{p-1}  (x+t)^{\alpha-p+1} t^{p-1}  \\
&\leq x^\alpha + \alpha^p t^\alpha + \alpha^p\sum_{i=1}^{p-1} x^{\alpha-i} t^i.
\end{align*}
Hence, it follows that
\begin{equation} \label{eq:Alpha_diff}
E\left[(S+Q)^\alpha - S^\alpha\right] \leq \alpha^p E[Q^\alpha] + \alpha^p\sum_{i=1}^{p-1} E[S^{\alpha-i} Q^i].
\end{equation}

To see that each of the expectations involving a product of $S$ and $Q$ is finite let ${\bf X} = (Q, N, C_1, C_2, \dots)$ and note that for $i = p-1$,
\begin{align} 
&E[S^{\alpha-p+1} Q^{p-1}] \\
&= E\left[ E\left[ \left. \left(  Q^{(p-1)/(\alpha-p+1)} \sum_{j=1}^N C_j R_j \right)^{\alpha-p+1}  \right|  {\bf X} \right]  \right] \notag \\
&\leq E\left[  \left(  E\left[ \left. Q^{(p-1)/(\alpha-p+1)} \sum_{j=1}^N C_j R_j   \right|  {\bf X} \right]  \right)^{\alpha-p+1} \right] \quad \text{(by Jensen's inequality)} \notag \\
&= (E[ R ])^{\alpha-p+1} E\left[  Q^{p-1} \left( \sum_{j=1}^N C_j    \right)^{\alpha-p+1} \right] , \label{eq:ProdMoments_1}
\end{align}
where the last equality was obtained by using the independence of $\{ R_j\}$ and ${\bf X}$. 

For $1 \leq i \leq p-2$ let $q_i = \lceil \alpha-i \rceil$ and condition on $Q$ and ${\bf X}$, respectively, to obtain
\begin{align}
E[S^{\alpha-i} Q^i] &= E\left[ \left( S^{\alpha-i} - \sum_{j=1}^N (C_j R_j)^{\alpha-i} \right) Q^i \right] + E\left[ Q^i \sum_{j=1}^N (C_j R_j)^{\alpha-i}  \right] \notag \\
&= E\left[ Q^i E\left[ \left. S^{\alpha-i} - \sum_{j=1}^N (C_j R_j)^{\alpha-i} \right| Q \right] \right] + E[ R^{\alpha-i}] E\left[ Q^i \sum_{j=1}^N C_j^{\alpha-i}  \right] \notag 
\end{align}
\begin{align}
&\leq E\left[ Q^i \left( E[R^{q_i - 1} | Q] \right)^{\frac{\alpha-i}{q_i - 1}} E\left[ \left. \left( \sum_{j=1}^N C_j \right)^{\alpha-i} \right| Q \right] \right] \\
&\hspace{5mm} + E[ R^{\alpha-i}] E\left[ Q^i \left( \sum_{j=1}^N C_j \right)^{\alpha-i}  \right] \notag \\
&= \left( \left( E[R^{q_i - 1}] \right)^{\frac{\alpha-i}{q_i - 1}} + E[R^{\alpha-i}] \right) E\left[    Q^i \left( \sum_{j=1}^N C_j \right)^{\alpha-i}  \right] , \label{eq:ProdMoments_2}
\end{align}
where for the inequality we used Lemma \ref{L.Alpha_Moments} ($\alpha-i > 1$) and the inequality $\sum_{i=1}^k y_i^\beta \leq \left( \sum_{i=1}^k y_i \right)^\beta$ for any $\beta \geq 1$ and $y_i \geq 0$.  Note that all the expectations involving $R$ in \eqref{eq:ProdMoments_1} and \eqref{eq:ProdMoments_2} are finite since $E[R^\beta] < \infty$ for all $0 < \beta < \alpha$ by assumption. Next, observe that all the other expectations are of the form $E\left[ Q^i \left( \sum_{j=1}^N C_j \right)^{\alpha-i} \right]$ for $1 \leq i \leq p-1$.  To see that these are finite use H\"older's inequality with $q = \alpha/(\alpha-i)$ and $r = \alpha/i$ to obtain
\begin{align*}
E\left[ Q^i \left( \sum_{j=1}^N C_j \right)^{\alpha-i} \right] &\leq \left|\left|\left( \sum_{j=1}^N C_j \right)^{\alpha-i} \right|\right|_q ||Q^i||_r \\
&= \left( E\left[ \left( \sum_{j=1}^N C_j \right)^{\alpha}   \right] \right)^{1/q} \left( E \left[ Q^\alpha \right] \right)^{1/r} < \infty.
\end{align*}
\end{proof}

%%%% Acknowledgements
\section*{Acknowledgements}
We would like to thank an anonymous referee and Matthias Meiners for their helpful comments.

\end{document}